\documentclass[12pt]{article}

 \usepackage{amsmath,amssymb,amscd,amsthm,esint}
 
\usepackage{graphics,amsmath,amssymb,amsthm,mathrsfs}

\usepackage{graphics,amsmath,amssymb,amsthm,mathrsfs,amsfonts,accents}
\usepackage{appendix}

\usepackage{sidecap}
\usepackage{float}
\usepackage{extarrows}
\usepackage{booktabs}
\usepackage{verbatim}
\usepackage{hyperref}
\usepackage[usenames,dvipsnames]{xcolor}

\newtheorem{thm}{Theorem}[section]
\newtheorem{lemma}[thm]{Lemma}
\newtheorem{cor}[thm]{Corollary}

\theoremstyle{definition}
\newtheorem{remark}[thm]{Remark}

\def\XXint#1#2#3{{\setbox0=\hbox{$#1{#2#3}{\int}$}
         \vcenter{\hbox{$#2#3$}}\kern-.5\wd0}}

\def\R{\mathbb{R}}
\def\C{\mathbb{C}}
\def\e{\varepsilon}

\def\Rd{\mathbb{R}^{d}}

\def\Cd{\mathbb{C}^d}
\def\Cdd{\mathbb{C}^{d\times d}}

\def\hW{\accentset{\circ}{W}}

\def\loc{{\rm loc}}

\def\H{\mathbb{H}}

\numberwithin{equation}{section}

\begin{document}

\title{Resolvent Estimates  in $L^\infty$  for the Stokes Operator \\ in  Nonsmooth Domains}

\author{Jun Geng  \thanks{Supported in part by NNSF grant 12371096.}
\qquad
Zhongwei Shen \thanks{Supported in part by NSF grant DMS-2153585.}}
\date{}

\maketitle

\begin{abstract}

We establish resolvent estimates in spaces of bounded solenoidal functions for the Stokes operator in 
a bounded domain $\Omega$ in $\R^d$ under the assumptions that $\Omega$ is $C^1$ for $d\ge 3$ and
Lipschitz for $d=2$.
As a corollary, it follows that the Stokes operator generates a uniformly bounded analytic semigroup
in the spaces of bounded solenoidal functions in $\Omega$.
The smoothness conditions on $\Omega$ are sharp.
The case of exterior domains with nonsmooth boundaries  is also studied.
The key step in our proof involves  new estimates that connect the pressure to the gradient of the  velocity in the $L^q$ average, but only on scales above certain level. 

\medskip


\medskip

\noindent {\it MR (2020) Subject Classification}: 35Q30.

\end{abstract}

\section{Introduction}

This paper studies the resolvent  problem for the  Stokes operator with the Dirichlet condition,
\begin{equation}\label{eq-0}
\left\{
\aligned
-\Delta  u +\nabla p +\lambda u & = F  & \quad & \text{ in } \Omega, \\
\text{\rm div}(u) & =0 & \quad  & \text{ in } \Omega,\\
u&=0 & \quad & \text{ on } \partial\Omega,
\endaligned
\right.
\end{equation}
where $\Omega$ is a domain in $\R^d$, $\lambda\in \Sigma_\theta $ is a  (given) parameter, and 
\begin{equation}\label{Sigma}
\Sigma_\theta
=\left\{
z\in \mathbb{C}\setminus \{ 0\}: \ |\text{arg} (z)|< \pi -\theta
\right\}
\end{equation}
for $\theta \in (0, \pi/2)$.
By the Lax-Milgram theorem it follows that for any $F\in  L^2(\Omega; \C^d)$, there exist a unique $u $ in  $W^{1, 2}_0(\Omega; \C^d)$ and a 
scalar function $p \in L^2_{\loc}(\Omega; \C)$ such that \eqref{eq-0} holds in the sense of distributions.
Moreover, the solution satisfies the energy estimate, 
\begin{equation}\label{e-est}
|\lambda| \| u \|_{L^2 (\Omega)} 
+  |\lambda|^{1/2} \| \nabla u \|_{L^2(\Omega)}
\le C \| F \|_{L^2(\Omega)},
\end{equation}
where $C$ depends on $d$ and $\theta$.
In this paper we shall be interested in the $L^q$ resolvent estimates for the Stokes operator,
\begin{equation}\label{p-est}
|\lambda| \| u \|_{L^q(\Omega)} 
+  |\lambda|^{1/2} \| \nabla u \|_{L^q(\Omega)} 
\le C \| F \|_{L^q(\Omega)}
\end{equation}
for $q\neq 2$.
Such estimates lead to the analyticity in the solenoidal $L^q$ spaces  of the Stokes semigroup, the solution operator to the nonstationary 
Stokes equations.
They play a crucial role in the study of the nonlinear Navier-Stokes equations in the domain $\Omega\times (0, T)$ and
have been investigated  extensively in the last forty years.

If $\Omega$ is a bounded or exterior domain with $C^{1, 1}$ boundary, the $L^q$ resolvent estimate \eqref{p-est},
together with estimates of $\nabla^2u$ and $\nabla p$ in $L^q(\Omega)$,
holds for $1< q< \infty$ \cite{Solo-1977, Giga-1981, Sohr-1987, Sohr-1994}.
See \cite{Sohr-1994} for a review as well as a comprehensive list of references in the case of smooth domains.
Also see \cite{Giga-1991, FKS-2005, Mitrea-2008, Tolksdorf-2018} and their references  for related work and applications to the nonlinear Navier-Stokes equations.
If $\Omega$ is a bounded Lipschitz domain, it was proved  in \cite{Shen-2012} by the second author of the present paper that
the resolvent estimate, 
\begin{equation}\label{est-Lip}
|\lambda| \| u \|_{L^q(\Omega)} \le C \| F \|_{L^q(\Omega)}
\end{equation}
 holds if $d\ge 3$ and
\begin{equation}\label{Lip-range}
\Big| \frac{1}{q}-\frac12 \Big|< \frac{1}{2d} +\e,
\end{equation}
where $\e>0$ depends on $\Omega$.
In particular, this shows  that the estimate \eqref{est-Lip} holds if  $(3/2)-\e< q< 3+\e$ and $\Omega$ is a bounded Lipschitz domain in $\R^3$.
Using an approach similar to that in \cite{Shen-2012}, F. Gabel and P. Tolksdorf  \cite{Tolksdorf-2022} were able to establish \eqref{est-Lip}
for $d=2$ and $(4/3)-\e< q< 4+\e$.
We mention that in \cite{Deuring-2001} P. Deuring constructed a bounded Lipschitz domain in $\R^3$ for which the estimate \eqref{est-Lip}
fails for large $q$.
Very recently, the present authors proved in \cite{GS-2023}  that if $\Omega$ is a bounded or exterior $C^1$ domain,
then the resolvent estimate \eqref{p-est} holds
for any $1< q< \infty$. 
In view of the work \cite{Deuring-2001}, the $C^1$ condition on the domain is more or less sharp for $d\ge 3$.
The approach used  in \cite{GS-2023} is based on a perturbation argument.
A key step is to establish the resolvent estimate  \eqref{p-est}  in the region above a Lipschitz graph,
\begin{equation}\label{H-1}
\H_\psi=
\big\{ (x^\prime, x_d) \in \R^d: \ x^\prime \in \R^{d-1} \text{ and } x_d > \psi (x^\prime) \big\}
\end{equation}
under the assumption that $\| \nabla^\prime \psi \|_\infty< c_0$ for some small $c_0>0$ depending on $d$, $\theta$ and $q$.

We now move on to the case $q=\infty$, which has been a long-standing open problem.
The closely related problem  regarding the analyticity of the Stokes semigroup  in spaces of bounded solenoidal functions
 was finally solved in \cite{AG-2013,AG-2014} by K. Abe and Y. Giga for smooth domains, using a blow-up argument.
 Previously, the result was only known for a half-space by explicit calculations \cite{Hieber-2001}.
 In \cite{AGH-2015}, using a more direct approach,
 the resolvent estimate \eqref{p-est} was established for $q=\infty$ and $\lambda \in \Sigma_\theta$  by K. Abe, Y. Giga, and M. Hieber.
 One of the key  steps in \cite{AG-2013, AG-2014, AGH-2015}  relies on a pressure estimate,
 \begin{equation}\label{pu-es}
\sup_{x\in \Omega}  \delta (x)  |\nabla p(x)|  \le C \| \nabla u \|_{L^\infty(\partial\Omega)},
 \end{equation}
 where $\delta(x) =\text{\rm dist}(x, \partial\Omega)$.
This estimate was proved  in \cite{AG-2013, AG-2014} for domains with $C^3$ boundaries.
It was noted in \cite{AGH-2015}  that the estimate holds in $C^{1, \alpha }$ domains  by using a result  obtained independently  in \cite{KLS-2013}
in the study of elliptic homogenization problems.
As a result, the resolvent estimates in \cite{AGH-2015} hold for domains with $C^2$ boundaries
(the other parts of the argument in \cite{AGH-2015} still require $C^2$).
Also see \cite{Hieber-2016, Hieber-2015} for related work on the Stokes semigroup in $L^\infty$ in smooth exterior domains.

In this paper we study the $L^\infty$ resolvent estimates for the Stokes operator in $C^1$ and Lipschitz domains.
Let 
\begin{equation}\label{L-sigma}
L^\infty_\sigma (\Omega)
=\left\{ F \in L^\infty(\Omega; \C^d):\
\text{\rm div} (F) =0 \text{ in } \Omega \text{ and } F \cdot n=0 \quad \text{ on } \partial\Omega \right\},
\end{equation}
where $n$ denotes the outward unit normal to $\partial\Omega$. 
The following  is  one of  the main results of the paper.

\begin{thm}\label{main-1}
Let $\Omega$ be a bounded $C^{1} $ domain   in $\Rd$, $d\ge 3$, or a bounded Lipschitz domain in $\R^2$.
Let  $\lambda \in \Sigma_\theta$, where $\theta \in (0, \pi/2)$.
Then for any $F\in L_\sigma^\infty(\Omega)$,  the solution of the Dirichlet problem \eqref{eq-0}   in 
$W_0^{1, 2}(\Omega; \Cd) \times L^2(\Omega; \C)$
satisfies the estimate,
\begin{equation}\label{est-0}
 |\lambda| \| u \|_{L^\infty (\Omega)}
\le C  \| F \|_{L^\infty (\Omega)}.
\end{equation}
Moreover, if $\Omega$ is a bounded $C^{1, \alpha}$ domain for some $\alpha>0$, then
\begin{equation}\label{est-1}
  |\lambda|^{1/2} \| \nabla u \|_{L^\infty(\Omega)}
  \le C \| F \|_{L^\infty(\Omega)}.
\end{equation}
The constants  $C$   in \eqref{est-0}-\eqref{est-1}  depend only on $d$, $\theta$ and $\Omega$.
\end{thm}

A few remarks are in order.

\begin{remark}
The smoothness conditions on $\Omega$ in Theorem \ref{main-1} are sharp.
In view of the counter-example given in \cite{Deuring-2001}, the $L^\infty$ resolvent estimate \eqref{est-0}  fails in a general Lipschitz domain for $d\ge 3$.
Also, the gradient estimate \eqref{est-1}  cannot be expected  in a $C^1$ domain, as there is no boundary Lipschitz estimate in $C^1$ domains
even for the Laplacian.
\end{remark}

\begin{remark}
It follows from Theorem \ref{main-1} by interpolation and duality that if $d=2$, the resolvent estimate \eqref{est-Lip}
holds in a Lipschitz domain for any $1< q< \infty$.
This extends the work \cite{Tolksdorf-2022}, which proved the estimate for $(4/3)-\e< q< 4+\e$ by 
a different method.
\end{remark}

\begin{remark}
If $\Omega$ is $C^1$ for $d\ge 3$ or Lipschitz for $d=2$,
our proof of Theorem \ref{main-1}  yields the estimate for the pressure,
\begin{equation}\label{pre-0}
|\lambda|^{1/2} \delta (x) |\nabla p(x)| \le C \| F \|_{L^\infty(\Omega)}
\end{equation}
for  any $x\in \Omega$ with $\delta (x) \ge |\lambda|^{-1/2}$.
Moreover, if $\Omega$ is $C^{1, \alpha}$ for some $\alpha>0$, then 
\begin{equation}
|\lambda|^{1/2} \| p \|_{B\!M\!O(\partial\Omega)}
\le C \| F \|_{L^\infty(\Omega)},
\end{equation}
 and \eqref{pre-0} holds for any $x\in \Omega$.
 In this case, since $p$ is harmonic in $\Omega$, by \cite{Fabes-1980},
  one may further improve \eqref{pre-0} to a Carleson measure estimate,
 \begin{equation}\label{Carleson}
 |\lambda| 
 \int_{B(x_0, r)\cap \Omega} \delta (x) |\nabla p(x)|^2\, dx
 \le C \| F \|^2_{L^\infty(\Omega)} r^{d-1}
 \end{equation}
 for any $x_0 \in \partial\Omega$ and $0< r< r_0$, where $r_0$ is given by \eqref{O}.
\end{remark}

Let $C_{0, \sigma}(\Omega)$ denote the closure of 
$$
C_{0, \sigma}^\infty (\Omega) =\left\{ F\in C_0^\infty(\Omega; \C^d):\  \text{\rm div}(F)=0 \text{ in } \Omega\right \}
$$
in $L^\infty(\Omega; \C^d)$.
If $\Omega$ is a bounded Lipschitz domain,
it is known that $$
C_{0, \sigma}(\Omega)=
\left\{ F\in C(\overline{\Omega}; \C^d): \ \text{\rm div}(F)=0 \text{ in } \Omega
\text{ and } F=0 \text{ on } \partial\Omega\right \}
$$
 \cite[Lemma 6.3]{AG-2013}.
For $F \in C_{0, \sigma}(\Omega)$ and $\lambda \in \Sigma_\theta$, define $u=R(\lambda)F$ to be the solution of \eqref{eq-0}.
By \cite[Theorem 1.1]{GS-2023},
  if $\Omega$ is  $C^1$, then $R(\lambda)F \in W^{1, p}_0 (\Omega; \C^d)$ for any $p>d$.
  Moreover, by the results in Section \ref{section-T}, 
  if $d=2$ and $\Omega$ is Lipschitz, we have $R(\lambda)F \in W^{1, p}_0(\Omega; \C^2)$ for some $p>2$.
  In both cases, we obtain  $R(\lambda)F\in C(\overline{\Omega}; \C^d)$ by Sobolev imbedding. As a result, 
it  follows from the estimate \eqref{est-0} that $\{ R(\lambda): \lambda \in \Sigma_\theta \}$ is a pseudo-resolvent in $C_{0, \sigma}(\Omega)$ and that 
there exists a closed operator $A$ such that $R(\lambda)=(\lambda -A)^{-1}$. 
As in \cite{AGH-2015},  we call $A$ the Stokes operator in $C_{0, \sigma}(\Omega)$.
The Stokes operator in $L^\infty_\sigma (\Omega)$ is defined in a similar manner.

\begin{cor}
Let $\Omega$ be a bounded $C^1$  domain in $\R^d$, $d\ge 3$, or a bounded Lipschitz domain in $\R^2$.
Then the Stokes operator $A$ generates a uniformly bounded analytic semigroup on $C_{0, \sigma}(\Omega)$
of angle $\pi/2$. It also generates a (non-$C_0$) uniformly bounded analytic semigroup on $L^\infty_\sigma (\Omega)$ of angle $\pi/2$.
\end{cor}

We now describe our approach to Theorem \ref{main-1}, which is based on the method used by K. Abe, Y. Giga,  and M. Hieber in  \cite{AGH-2015}
for $L^\infty$ resolvent estimates in smooth domains.
This method was inspired by the works of K. Masuda \cite{Masuda} and H.B. Stewart \cite{Stewart} for general elliptic operators
in the space of continuous functions.
The basic idea is to localize the $L^q$ resolvent estimates for some $q>d$ and use the estimate \eqref{pu-es}  to control the error terms involving the pressure.
The major issue with this method for a nonsmooth domain is the lack  of  $L^\infty$ estimates   for   $\nabla u$,
as both $\nabla u$ and $p$ are not expected to be bounded up to the boundary in a $C^1$ domain. 
Moreover, the main set-up  form for a priori estimate, 
$$
\aligned
& M_q(u, p)(\lambda)\\
& =|\lambda| |u(x)|
+ |\lambda|^{1/2} |\nabla u(x)| 
+ |\lambda|^{d/2q} \|\nabla^2 u \|_{L^q(\Omega \cap B(x, |\lambda|^{-1/2}))}
+ |\lambda|^{d/2q} \|\nabla p \|_{L^q(\Omega\cap B(x, |\lambda|^{-1/2}))},
\endaligned
$$
 used in \cite{AGH-2015},  involves $\nabla^2 u$ and $\nabla p$, which cannot be controlled even for $C^{1, \alpha}$ domains.
 
We manage to resolve  these issues by replacing \eqref{pu-es} with several  new and stronger  estimates that  relate the pressure to the gradient of the velocity
 in the $L^q$ average, but  only on scales above a given level. 
A key observation is that such estimates are only needed on scales above $|\lambda|^{-1/2}$ when it is used to bound  the $L^\infty$ norm  of $u$
by an interpolation  inequality. 
With the $L^q$ resolvent estimates we obtained  recently in \cite{GS-2023} for $1< q<\infty$, 
this allows us to treat the case of $C^1$ domains without the Lipschitz bound for the velocity.
In the case $d=2$, the approach, together with a perturbation argument, which upgrades the $L^2$ estimates for $\nabla u$ to $L^q$ estimates for some $q>2$,
yields the $L^\infty$ resolvent estimate for Lipschitz domains.
We mention that  our new estimates for the pressure, which are of independent interest,  are proved for Lipschitz domains.
By taking limits in the scale,  it follows that the estimate \eqref{pu-es}  used in \cite{AG-2013, AG-2014, AGH-2015} in fact holds for Lipschitz domains.
However, as we pointed out above, it is not strong enough  for treating nonsmooth domains.

In the following we give a sketch of the proof of Theorem \ref{main-1} and provide a more detailed 
description of  the new estimates mentioned above.

\medskip

Step 1. We localize the generalized $L^q$ resolvent estimates obtained in \cite{GS-2023}  for $1< q< \infty$ in a graph domain $\H_\psi$, given by \eqref{H-1},
where $\psi: \R^{d-1} \to \R$ is a Lipschitz function with small Lipschitz norm $\|\nabla^\prime \psi\|_\infty$.
This, together with local regularity estimates for the Stokes equations (with $\lambda=0$) in $C^1$ domains, allows us to show that for any $N\ge 2$
and $d< q< \infty$,
\begin{equation}\label{V-00}
\aligned
  |\lambda| \| u \|_{L^\infty(\Omega)}
 & \le C  N^{4d+8} \| F \|_{L^\infty(\Omega)}  
+ C N^{\frac{d}{q}-1}
 |\lambda| \| u \|_{L^\infty(\Omega)}\\
& + C N^{\frac{d}{q}-1}
|\lambda|^{\frac12}
\sup_{\substack{ t< r< r_0\\ B(x_0, 2r ) \subset \Omega}}
\inf_{\substack {\beta \in \C }}
\left(\fint_{B(x_0, r ) } | p -\beta|^q \right)^{1/q}\\
&+ C N^{\frac{d}{q}-1}
|\lambda|^{\frac12}
\sup_{\substack{ x_0\in \partial \Omega \\ t< r< r_0 } }
\inf_{\substack {\beta \in \C }}
\left\{
\left(\fint_{B(x_0, r )\cap \Omega } | p -\beta|^q \right)^{1/q}
+ \left(\fint_{B(x_0, r  )\cap \partial\Omega}
|p -\beta|^q \right)^{1/q} \right\},
\endaligned
\end{equation}
where $t=|\lambda|^{-1/2}$, $\Omega$ is $C^1$,  and $C$ depends only on $d$, $q$ and $\Omega$.
See Sections \ref{section-L} and \ref{section-V}.

\medskip

Step 2.
To bound the terms in \eqref{V-00} involving the pressure $p$, we study the Neumann problem for Laplace's equation,
\begin{equation}\label{NP-00}
\left\{
\aligned
\Delta \phi & =0 & \quad & \text{ in } \Omega,\\
\frac{\partial \phi}{\partial n}  & = (n_i \partial_j -n_j \partial_i ) g_{ij} & \quad & \text{ on } \partial \Omega,
\endaligned
\right.
\end{equation}
where $\Omega$ is a bounded Lipschitz domain and the repeated indices $i, j$ are summed from $1$ to $d$.
Let $g=(g_{ij})$ and  $I(x_0, r)= B(x_0, r) \cap \partial\Omega$.
For $0< r<   r_0$, we introduce a new  maximal function, 
\begin{equation}\label{M00}
M_q (g)  (r) =\sup_{\substack{x_0 \in \partial\Omega \\ r\le  s< r_0}}
\left(\fint_{I(x_0, s)} |g|^q \right)^{1/q},
\end{equation}
which measures the $L^q$ average  of $g$ above the scale $r$.
Let $2\le q<\infty$ and $0< r< r_0$.
We  show that  
\begin{equation}\label{P00}
\sup_{B(x, 2r)\subset \Omega}
\left(\fint_{B(x, r)} |\phi - \fint_{B(x, r)} \phi |^q \right)^{1/q}
\le C M_q (g) (r) ,
\end{equation}
and that  for $x_0\in \partial\Omega$,
\begin{equation}\label{P000}
\aligned
& \left(\fint_{B(x_0, r) \cap \partial\Omega}
|\phi -\fint_{B(x_0, r) \cap \partial\Omega} \phi |^q \right)^{1/q}\\
&\qquad+
\left(\fint_{B(x_0, r) \cap \Omega}
|\phi -\fint_{B(x_0, r) \cap\partial \Omega} \phi  |^q \right)^{1/q}
\le C  M_q (g) (r), 
\endaligned
\end{equation}
where $C$ depends only on $d$, $q$ and $\Omega$.
See Sections \ref{section-N} and \ref{section-Ne}.
Since $\phi$ is harmonic, it follows from \eqref{P00} that
\begin{equation}\label{Pa}
\delta(x)  |\nabla \phi (x)|
\le C M_q (g) (\delta (x))
\end{equation}
for any $x\in \Omega$, where $\delta(x) =\text{\rm dist}(x, \partial\Omega)$.
In particular, this implies that
\begin{equation}
\sup_{x\in \Omega} \delta(x)  |\nabla \phi (x)|
\le C \| g \|_{L^\infty(\partial\Omega)},
\end{equation}
which gives \eqref{pu-es} if $g_{ij} =\partial_j u_i$.
It is also interesting to point out that the estimate \eqref{P000} implies
\begin{equation}
\| \phi \|_{B\!M\!O(\partial\Omega)}
\le C \| g \|_{L^\infty(\partial\Omega)},
\end{equation}
which is sharp, as $\phi$ is related to $g$ by some singular integral operator on $\partial\Omega$.

\medskip

Step 3.
As in \cite{AG-2013, AG-2014, AGH-2015},  note that if $(u, p)$ is a smooth solution of \eqref{eq-0} with $F\in C_{0, \sigma}^\infty(\Omega)$,
then $\Delta p=0$ in $\Omega$  and
$$
\nabla p \cdot n = (n_i \partial_j  -n_j \partial_i ) \partial_j u_i \quad \text{ on } \partial\Omega.
$$
As a result, we may use  estimates \eqref{P00}-\eqref{P000} to bound the last two terms in the right-hand side of \eqref{V-00}
and obtain 
\begin{equation}\label{est-a}
|\lambda| \| u \|_{L^\infty(\Omega)}
\le C N^{4d+8} \| F \|_{L^\infty(\Omega)}
+ CN^{\frac{d}{q}-1} |\lambda| \| u \|_{L^\infty(\Omega)}
+ C N^{\frac{d}{q}-1} |\lambda|^{1/2}
M_q (|\nabla u |) (t),
\end{equation}
where $t=|\lambda|^{-1/2}$.
By using local regularity estimates for the Stokes equations in $C^1$ domains, we deduce that for $t=|\lambda|^{-1/2}$,
\begin{equation}\label{est-b}
M_q (|\nabla u|) (t ) \le C \left\{ |\lambda|^{1/2} \| u \|_{L^\infty(\Omega)}
+|\lambda|^{-1/2} \| F \|_{L^\infty(\Omega)} \right\},
\end{equation}
where $C$ depends on $d$, $q$ and $\Omega$. See Lemma \ref{lemma-C-2}.
By combining \eqref{est-a} with \eqref{est-b}, we see that 
\begin{equation}\label{est-c}
|\lambda| \| u \|_{L^\infty(\Omega)}
\le C_0 N^{4d+8} \| F \|_{L^\infty(\Omega)}
+ C_0N^{\frac{d}{q}-1} |\lambda| \| u \|_{L^\infty(\Omega)}.
\end{equation}
To finish the proof of \eqref{est-0}, we fix $q>d$ and choose $N$ so large that $C_0 N^{\frac{d}{q}-1} \le (1/2)$.

\medskip

Step 4. If $\Omega$ is $C^{1, \alpha}$ for some $\alpha>0$,
the local estimates for the Stokes equations give
\begin{equation}\label{est-g}
|\lambda|^{1/2} \|\nabla u \|_{L^\infty(\Omega)}
\le C \left\{
|\lambda| \|  u \|_{L^\infty(\Omega)}
+  \| F \|_{L^\infty(\Omega)} \right\}
\end{equation}
for $|\lambda|> r_0^{-2}$.
See Remark \ref{re-smooth}.
The gradient estimate \eqref{est-1} follows from \eqref{est-0} and \eqref{est-g}.

\medskip

Step 5. A careful inspection of the argument outlined above indicates that the $L^\infty$ resolvent estimate \eqref{est-0} holds in $\Omega$,
as long as (1)  one is able to establish the generalized $L^q$ resolvent estimates for $u$ and $\nabla u$ for some $q>d$;
(2) the local boundary $L^q$ estimate holds for the Stokes equations (with $\lambda=0$)
 in $\Omega$ for some $q>d$, which is needed for bounding 
$M_q(|\nabla u|)(|\lambda|^{-1/2})$.
These are indeed possible for Lipschitz domains in $\R^2$ by using a perturbation argument to upgrade the $L^2$ estimates for $\nabla u$.
See Section \ref{section-T}.

\medskip

The approach we use for  bounded domains  works equally well for graph domains.
It also works well for exterior domains,  provided that  $|\lambda|$ is sufficiently large. 
We call $\Omega$ an exterior Lipschitz (or $C^1$, resp.)
domain if  $\Omega$ is open, connected and $\overline{\Omega}^c$ is a nonempty bounded  set
with Lipschitz (or $C^1$, resp.) boundary.

\begin{thm}\label{main-2}
Let $\Omega$ be an exterior $C^1$ domain in $\R^d$ for $d\ge 3$ or an exterior Lipschitz domain in $\R^2$.
Let $\lambda\in \Sigma_\theta$, where  $\theta \in (0, \pi/2)$.
There exists $\lambda_0>1$, depending only on $d$, $\theta$ and $\Omega$, with the property that
for any $\lambda\in \Sigma_\theta$ with $|\lambda|> \lambda_0$ and for any $F\in L^\infty_\sigma (\Omega)$,
there exist a unique $u\in W^{1, 2}_{\loc} (\Omega;  \C^d) \cap L^\infty(\Omega; \C^d)$
and $ p\in L^2_{\loc} (\Omega; \C)$ such that \eqref{eq-0} holds in
the sense of distributions and  that $p(x)=\beta + o(1)$ as $|x|\to \infty$ for some $\beta\in \C$.
Moreover, the solution satisfies the estimate \eqref{est-0}.
Furthermore,  if $\Omega$ is an exterior $C^{1, \alpha}$  domain for some $\alpha>0$, then the estimate \eqref{est-1}
holds. 
\end{thm}

\begin{cor}\label{cor-ext}
Let $\Omega$ be an exterior  $C^1$  domain in $\R^d$, $d\ge 3$ or an exterior Lipschitz domain in $\R^2$.
Then the Stokes operator $A$ generates an analytic semigroup on $C_{0, \sigma}(\Omega)$ and $L^\infty_\sigma (\Omega)$.
\end{cor}

Since the  $L^\infty$ resolvent estimates are only proved for $|\lambda|> \lambda_0$ in Theorem \ref{main-2},
it is not known that  the Stokes semigroup in Corollary \ref{cor-ext}
is uniformly bounded.
In the case of smooth exterior domains, the results on the boundedness can be found in \cite{Hieber-2015, Hieber-2016}.

We end this section by recalling the definition of a bounded Lipschitz (or $C^1$, resp.) domain.
Let  $\Omega$ be a bounded domain in $\R^d$.
We say $\Omega$ is Lipschitz (or $C^1$, resp.) if there exists $r_0>0$ such that for any $x_0 \in \partial\Omega$,
\begin{equation}\label{O}
 \Omega\cap B(x_0, r_0)
=D \cap B(x_0, r_0) \quad \text{ and } \quad
\partial \Omega \cap B(x_0, r_0) =\partial D \cap B(x_0, r_0),
\end{equation}
where $D=x_0+ O \H_\psi$ for some $d\times d$ orthogonal matrix $O$ and some Lipschitz  (or $C^1$, resp.)
function $\psi$ in $\R^{d-1}$
with $\psi  (0)=0$.

%\noindent{\bf Acknowledgement.}
% The authors thank the anonymous referees for their helpful comments that improved the quality of the manuscript.

%%%%%%%%%%%%%%%%%%%%%%%%%%%%%%%%%%%%%%%%%%%%%%%%%%%%%%%%

\section{Localization}\label{section-L}

Let $1< q< \infty$. 
In this section we carry out a localization procedure, using $L^q$ resolvent estimates  in a region above a Lipschitz graph with small Lipschitz norm, 
obtained in \cite{GS-2023}.

Let 
\begin{equation}\label{H}
\H_\psi= \left\{ (x^\prime, x_d) \in \Rd: \  x^\prime\in \mathbb{R}^{d-1} \text{ and } x_d > \psi (x^\prime) \right\}, 
\end{equation}
where $\psi: \mathbb{R}^{d-1} \to \mathbb{R}$ is a Lipschitz function and $\psi (0)=0$.
Consider   the generalized  resolvent problem for the Stokes equations,
\begin{equation}\label{G-eq}
\left\{
\aligned
-\Delta u +\nabla p+\lambda u  & =F +\text{\rm div} (f) & \quad & \text{ in } \H_\psi,\\
\text{\rm div}(u) & = g & \quad & \text{ in } \H_\psi,\\
u & =0 & \quad  & \text{ on } \partial \H_\psi,
\endaligned
\right.
\end{equation}
 where $\lambda\in \Sigma_\theta$.
 For $1< q< \infty$ and $q^\prime =\frac{q}{q-1}$, define
\begin{equation}\label{AB}
\aligned
A_\psi^q =L^q(\H_\psi; \C)  + \hW^{1, q}(\H_\psi; \C) \quad
\text{ and } \quad
B_\psi^q = L^q(\H_\psi; \C) \cap \hW^{-1, q}(\H_\psi; \C),
\endaligned
\end{equation}
where
$$
\hW^{1, q}(\H_\psi; \C)=\left\{  u\in L^q_{\loc}({\H_\psi}; \C): \ \nabla u \in L^q(\H_\psi; \Cd ) \right\},
$$
with the norm $\|\nabla u \|_{L^q(\H_\psi)}$, 
and $\hW^{-1, q}(\H_\psi; \C)$ denotes the dual of $\hW^{1, q^\prime} (\H_\psi;  \C)$.
Observe  that $A_\psi^q$ and $B_\psi^q=\left( A_\psi^{q^\prime} \right)^\prime$ are Banach spaces with the usual norms,
$$
\| p \|_{A_\psi^q}
=\inf \left\{ \| p_1 \|_{L^q(\H_\psi)} + \| \nabla p_2 \|_{L^q(\H_\psi)}: \ \  p=p_1 + p_2  \text{ in } \H_\psi \right\}
$$
and
$$
\| g \|_{B_\psi^q}
= \| g \|_{L^q(\H_\psi)} + \| g \|_{\hW^{-1, q} (\H_\psi)}.
$$

Let $\nabla^\prime\psi$ denote the gradient of $\psi$ in $\R^{d-1}$.
 The following theorem was proved  by the present authors  in \cite[Theorem 4.1]{GS-2023}.

\begin{thm}\label{thm-G}
Let $ \lambda\in \Sigma_\theta$ and $1< q<\infty$.
There exists $c_0\in (0, 1)$,  depending only on $d$, $q$ and $\theta$, such that 
if  $\|\nabla^\prime \psi \|_\infty\le  c_0$, then 
for any $F\in L^q(\H_\psi; \Cd)$, $f\in L^q(\H_\psi; \Cdd)$ and $g \in B_\psi^q$,
there exists  a unique  $(u, p)$ such that $u \in W_0^{1, q}(\H_\psi; \Cd)$, 
$p\in A_\psi^q $, and \eqref{G-eq} holds.
Moreover, the solution satisfies 
\begin{equation}\label{G-est}
\aligned
 & |\lambda|^{1/2} \| \nabla u \|_{L^q(\H_\psi)}
+|\lambda| \| u \|_{L^q(\H_\psi)}
\\
&\quad \le C \left\{
\| F \|_{L^q(\H_\psi)}
+ |\lambda|^{1/2}  \| f \|_{L^q(\H_\psi)}
+|\lambda|^{1/2}  \| g \|_{L^q(\H_\psi)}
+  |\lambda| \| g \|_{\hW^{-1, q}(\H_\psi)}
\right\},
\endaligned
\end{equation}
where $C$ depends only on $d$, $q$ and $\theta$.
\end{thm}

For $r>0$, let
\begin{equation}
D_r= B(0, r) \cap \H_\psi \quad \text{ and } \quad
I_r = B(0, r) \cap \partial \H_\psi.
\end{equation}

\begin{thm}\label{thm-L}
Let $ \lambda\in \Sigma_\theta$ and $2< q<\infty$.
Suppose $\|\nabla^\prime \psi \|_\infty\le c_0$, where $c_0=c_0(d, q, \theta)$ is given by Theorem \ref{thm-G}.
Let $(u, p)\in W^{1, q} (D_{Nr}; \C^d) \times L^q(D_{Nr}; \C)$ be a weak solution of
\begin{equation}\label{L-11}
\left\{
\aligned
-\Delta u +\nabla p +\lambda u & = F &\quad & \text{ in } D_{Nr},\\
\text{\rm div} (u) & =0 & \quad & \text{ in } D_{Nr},\\
u& =0 & \quad & \text{ on } I_{Nr},
\endaligned
\right.
\end{equation}
where $N\ge 2$ and $F\in L^\infty(D_{Nr}; \C^d)$. 
Then 
\begin{equation}\label{L-12}
\aligned
&  |\lambda| \| u \|_{L^q(D_r)} 
+|\lambda|^{\frac12} \| \nabla u \|_{L^q(D_r)}\\
  & \le 
C \Bigg\{
\| F \|_{L^\infty(D_{Nr})}  (Nr)^{\frac{d}{q}}
+ \left(\fint_{D_{Nr}} |u|^q\right)^{1/q}
\left(
(Nr)^{\frac{d}{q}-2} + (Nr)^{\frac{d}{q}-1} |\lambda|^{\frac12} \right)\\
 & \quad + \left(\fint_{D_{Nr}} ( |\nabla u| + |p|)^q \right)^{1/q} (Nr)^{\frac{d}{q}-1}
+\left(\fint_{I_{Nr}} ( |\nabla u| + |p|)^q \right)^{1/q}
(Nr)^{\frac{d}{q}-1}
\Bigg\},
\endaligned
\end{equation}
where $C$ depends only on $d$, $q$ and $\theta$.
\end{thm}

\begin{proof}
We begin by choosing a cut-off function $\varphi \in C_0^\infty (B(0, Nr))$ such that 
$\varphi =1$ in $B(0, r)$,
\begin{equation}\label{L-13a}
\|\nabla \varphi \|_\infty \le C (Nr)^{-1} \quad \text{ and } \quad
\|\nabla^2 \varphi \|_\infty \le C (Nr)^{-2},
\end{equation}
where $C$ depends only on $d$.
Note that 
\begin{equation}\label{L-13}
\left\{
\aligned
-\Delta (u \varphi) + \nabla (p \varphi) + \lambda (u \varphi) & = h & \quad & \text{ in } \H_\psi,\\
\text{\rm div} (u \varphi) & = g & \quad & \text{ in } \H_\psi,\\
u \varphi & =0  & \quad & \text{ on } \partial \H_\psi,
\endaligned
\right.
\end{equation}
where
\begin{equation}\label{L-14}
\left\{
\aligned
h & =F \varphi - 2 (\nabla u) (\nabla \varphi) - u \Delta \varphi + p (\nabla \varphi),\\
g & =u \cdot \nabla \varphi.
\endaligned
\right.
\end{equation}
It follows by Theorem \ref{thm-G} that
\begin{equation}\label{L-15}
\aligned
   |\lambda| \| u  & \|_{L^q(D_r)} 
+|\lambda|^{\frac12} \| \nabla u \|_{L^q(D_r)}\\
 &\le  |\lambda| \| u \varphi \|_{L^q(\H_\psi)} + 
|\lambda|^{\frac12} \| \nabla ( u\varphi) \|_{L^q(\H_\psi )}\\
 & \le 
C \left\{ \| h \|_{L^q(\H_\psi)}
+ |\lambda|^{\frac12} \| g \|_{L^q(\H_\psi)}
+ |\lambda|  \| g \|_{\hW^{-1, q}(\H_\psi)} \right\},
\endaligned
\end{equation}
where we have used the fact $\varphi=1$ in $B(0, r)$ for the first inequality.

We now proceed to bound the right-hand side of \eqref{L-15} by the right-hand side of \eqref{L-12}.
Using \eqref{L-13a}, 
it  is not hard to see that
$\| g \|_{L^q(\H_\psi)} \le C  \| u \|_{L^q(D_{Nr})} (Nr)^{-1}$,
 and that 
$$
\aligned
\| h \|_{L^q(\H_\psi)}
& \le C \Big\{
\| F \|_{L^\infty (D_{Nr})}  (Nr)^{\frac{d}{q}}
+ \|\nabla u \|_{L^q(D_{Nr})} (Nr)^{-1}\\
& \qquad\qquad
 + \| p \|_{L^q(D_{Nr})} (Nr)^{-1} 
+ \| u \|_{L^q(D_{Nr})} (Nr)^{-2} \Big\}.
\endaligned
$$
In view of \eqref{L-12},
this gives the desired estimates for the first two terms in the right-hand side of \eqref{L-15}.

It remains to bound the third term $|\lambda| \| g \|_{\hW^{-1, q}(\H_\psi)}$.
To this end,  let $\eta \in \hW^{1, q^\prime}(\H_\psi)$ with $\| \nabla \eta \|_{L^{q^\prime} (\H_\psi)}\le 1$. 
Note that  for any $\beta \in \C$, 
$$
\aligned
\lambda \int_{\H_\psi}  g   \eta
&=\lambda \int_{\H_\psi} \text{\rm div} (u \varphi)   (\eta-\beta)
=\lambda \int_{\H_\psi} (u \cdot \nabla \varphi) (\eta-\beta)\\
&=\int_{\H_\psi}
\{ (F+\Delta u -\nabla p)\cdot \nabla \varphi  \} (\eta -\beta)\\
&=J_1 +J_2,
\endaligned
$$
where we have used \eqref{L-11}.
Let $\beta =\fint_{D_{2Nr}} \eta$ .
Clearly, by H\"older's inequality, 
$$
\aligned
|J_1|
 & =\big|\int_{\H_\psi} (F \cdot \nabla \varphi) ( \eta-\beta)\big|\\
&\le \| F \|_{L^\infty(D_{Nr})}  \|\nabla \varphi \|_{L^q(D_{Nr})}
\| \eta -\beta \|_{L^{q^\prime}(D_{Nr})}.
\endaligned
$$
By applying a Poincar\'e inequality, we obtain 
\begin{equation}
\aligned
|J_1|
 & \le C \| F \|_{L^\infty (D_{Nr})} 
\|\nabla \varphi \|_{L^q(D_{Nr})} \|\nabla \eta \|_{L^{q^\prime}(D_{2Nr})} (Nr) \\
& \le C \| F \|_{L^\infty (D_{Nr})}  (Nr)^{\frac{d}{q}},
\endaligned
\end{equation}
where we have used the fact $\|\nabla \eta \|_{L^{q^\prime}(\H_\psi)} \le 1$.
Next, we use integration by parts to obtain 
$$
\aligned
J_2 & =\int_{\H_\psi} (  (\Delta u-\nabla p) \cdot \nabla\varphi) (\eta-\beta)\\
&=-\int_{\H_\psi} (\partial_j u^i -\delta_{ij} p ) (\partial^2_{ij}  \varphi )  (\eta-\beta)
-\int_{\H_\psi} (\partial_j  u^i -\delta_{ij} p ) (\partial_i  \varphi )(\partial_j \eta)\\
&\qquad\qquad\qquad
 +\int_{\partial H_\psi}  n_j (\partial_j u^i -\delta_{ij} p) ( \partial_i \varphi ) (\eta -\beta)\\
 &=J_{21} +J_{22}+ J_{23},
\endaligned
$$
where $n=(n_1, \dots, n_d)$ denotes the outward unit normal to $\partial \H_{\psi}$ and  
the repeated indices $i, j$ are summed from $1$ to $d$.
We point out that the use of integration by parts  in $J_2$ can be justified by an approximation argument with the observation  \eqref{B1-5}.

Again,  by H\"older's inequality, 
$$
\aligned
|J_{21} +J_{22} |
&\le  C \| |\nabla u | +|p| \|_{L^q(D_{Nr})} \| \eta-\beta \|_{L^{q^\prime}(D_{2Nr})} (Nr)^{-2}\\
& \qquad \qquad
+ C  \| |\nabla u | +|p| \|_{L^q(D_{Nr})} \| \nabla \eta  \|_{L^{q^\prime}(D_{2Nr})} (Nr)^{-1}\\
& \le C  \| |\nabla u | +|p| \|_{L^q(D_{Nr})} (Nr)^{-1},
\endaligned
$$
where  $\beta= \fint_{D_{2Nr} } \eta$ and we have applied  a Poincar\'e inequality.
Similarly,
$$
\aligned
|J_{23} |
& \le C \| |\nabla u| + |p| \|_{L^q(I_{Nr})} \|\eta-\beta \|_{L^{q^\prime}(I_{2Nr})} (Nr)^{-1}\\
& \le C \| |\nabla u | + |p| \|_{L^q(I_{Nr})}  \|\nabla \eta \|_{L^{q^\prime} (\H_\psi)} (Nr)^{-\frac{1}{q^\prime}},
\endaligned
$$
where we have used the Poincar\'e inequality, 
\begin{equation}\label{P-ineq}
\left(\int_{I_R} |\eta -\beta|^{q^\prime} \right)^{1/q^\prime}
\le C  R^{1-\frac{1}{q^\prime}} \left(\int_{D_{2R}} |\nabla \eta|^{q^\prime} \right)^{1/q^\prime},
\end{equation}
with $\beta =\fint_{D_{2R}} \eta $.
See Remark \ref{re-L1} below  for a proof of \eqref{P-ineq}. 
As a result, we have proved that
$$
\aligned
|J_1 + J_2|
& \le C \| F\|_{L^\infty(D_{Nr})}  (Nr)^{\frac{d}{q}}
+ C \| |\nabla u| + |p | \|_{L^q(D_{Nr})} (Nr)^{-1}\\
&\qquad\qquad
+ C \| |\nabla u| + |p | \|_{L^q(I_{Nr}) } (Nr)^{-1 + \frac{1}{q}}.
\endaligned
$$
By duality, we obtain 
$$
\aligned
|\lambda | \| g \|_{\hW^{-1, q}(\H_\psi)}
& \le C \| F\|_{L^\infty(D_{Nr})}  (Nr)^{\frac{d}{q}}
+ C \| |\nabla u| + |p | \|_{L^q(D_{Nr})} (Nr)^{-1}\\
&\qquad\qquad
+ C \| |\nabla u| + |p | \|_{L^q(I_{Nr}) } (Nr)^{-1 + \frac{1}{q}}.
\endaligned
$$
This, together with \eqref{L-15} as well as the estimates for $\| g \|_{L^q(\H_\psi)}$  and $\| h \|_{L^q(\H_\psi)}$,
completes  the proof of  \eqref{L-12}.
\end{proof}

\begin{remark}\label{re-L1}
The  Poincar\'e inequality \eqref{P-ineq},  with $\beta =\fint_{D_{2R}} \eta$ and $1<q^\prime< \infty$,
holds for any Lipschitz graph domain $\H_\psi$. The smallness 
condition $\|\nabla^\prime \psi \|_\infty\le c_0$ is not needed.
To see this, by dilation, it suffices to consider the case $R=1$.
Let $\varphi \in C_0^\infty(B(0, 2))$ such that  $0 \le \varphi \le 1$  and $\varphi =1$ on $B(0, 1)$.
Note that $n_d =- ( 1+|\nabla^\prime \psi|^2 )^{-1/2}$ on $\partial \H_{\psi}$.
Using integration by parts, we have 
$$
\aligned
(1+ \|\nabla^\prime \psi \|_\infty^2)^{-1/2} 
 \int_{I_1}
|\eta-\beta|^{q^\prime}
 & \le  - \int_{\partial D_2} n_d \varphi | \eta -\beta|^{q^\prime}\\
 & =- \int_{D_2} \partial_d \varphi | \eta-\beta|^{q^\prime}
 -  \int_{D_2} \varphi \partial_d |\eta -\beta|^{q^\prime}\\
 & \le C \int_{D_2} |\eta-\beta|^{q^\prime}
 + C \int_{D_2} |\eta-\beta|^{q^\prime-1} |\nabla \eta|\\
 &\le C \int_{D_2} |\eta -\beta|^{q^\prime}
 + C \int_{D_2} |\nabla \eta|^{q^\prime}\\
 &\le C \int_{D_2} |\nabla \eta|^{q^\prime},
\endaligned
$$
where we have used H\"older's inequality as well as the Poincar\'e inequality 
$\| \eta -\beta \|_{L^{q^\prime} (D_2)} \le C \| \nabla \eta \|_{L^{q^\prime}(D_2)}$.
\end{remark}

The next theorem provides the generalized resolvent estimates for the Stokes equations in $\R^d$.
See \cite[Theorem 2.1] {GS-2023} for a proof.

\begin{thm}\label{R-thm-1} 
Let  $1<q<\infty$ and $\lambda\in \Sigma_\theta $.
For any $F\in L^q(\Rd; \Cd)$,  $f\in L^q(\Rd; \Cdd)$, and $g\in L^q(\Rd; \C)\cap \hW^{-1, q}(\Rd; \C)$,
there exists  a unique $u\in W^{1, q}(\Rd; \Cd) $ such that 
\begin{equation}\label{R-eq}
\left\{
\aligned
-\Delta u +\nabla p +\lambda u & = F +\text{\rm div} (f) ,\\
\text{\rm div} (u) & =g 
\endaligned
\right.
\end{equation}
hold in $\Rd$ for some $p\in L^1_{\loc}(\Rd; \C)$ in the sense of distributions.
Moreover, the solution  satisfies the estimate,
\begin{equation}\label{R-est-1}
\left\{
\aligned
|\lambda |^{1/2} \| \nabla u \|_{L^q (\Rd)} 
 & \le C \left\{ \| F \|_{L^q(\Rd)} + |\lambda|^{1/2} \| f \|_{L^q(\Rd)} + |\lambda|^{1/2} \| g \|_{L^q(\Rd)} \right\}, \\
|\lambda|
\| u \|_{L^q(\Rd)}
& \le C \left\{   \| F \|_{L^q(\Rd)} + |\lambda|^{1/2} \| f \|_{L^q(\Rd)} +  |\lambda | \| g \|_{\hW^{-1, q} (\Rd)} \right\}, 
\endaligned
\right.
\end{equation}
and $p\in L^q(\Rd; \C) + \hW^{1, q} (\Rd; \C)$, 
where $C$ depends on $d$, $q$ and $\theta$.
\end{thm}

Let $B_r=B(0, r)$. The following theorem is the interior version of Theorem \ref{thm-L}.

\begin{thm}\label{thm-L2}
Let $ \lambda\in \Sigma_\theta$ and $2< q<\infty$.
Let $(u, p)\in W^{1, q} (B_{Nr} ; \C^d) \times L^q(B_{Nr} ; \C)$ be a weak solution of
\begin{equation}\label{L-21}
\left\{
\aligned
-\Delta u +\nabla p +\lambda u & = F &\quad & \text{ in } B_{Nr} ,\\
\text{\rm div} (u) & =0 & \quad & \text{ in } B_{Nr}  ,\\
\endaligned
\right.
\end{equation}
where $N\ge 2$ and $F\in L^\infty(B_{Nr} ; \C^d)$. 
Then 
\begin{equation}\label{L-22}
\aligned
&  |\lambda| \| u \|_{L^q(B_r) } 
+|\lambda|^{\frac12} \| \nabla u \|_{L^q(B_r )}\\
  & \le 
C \Bigg\{
\| F \|_{L^\infty (B_{Nr})}  (Nr)^{\frac{d}{q}}
+ \left(\fint_{B_{Nr}} |u|^q\right)^{1/q}
\left(
(Nr)^{\frac{d}{q}-2} + (Nr)^{\frac{d}{q}-1} |\lambda|^{\frac12} \right) \\
 & \quad\qquad + \left(\fint_{B_{Nr}} ( |\nabla u| + |p|)^q \right)^{1/q} (Nr)^{\frac{d}{q}-1}
\Bigg\},
\endaligned
\end{equation}
where $C$ depends only on $d$, $q$ and $\theta$.
\end{thm}

\begin{proof}
The proof is similar to and slightly easier than  that of Theorem \ref{thm-L}.
Choose a cut-off function $\varphi\in C_0^\infty (B_{Nr})$ with the properties  \eqref{L-13a}.
Note that  $u \varphi$ satisfies the resolvent Stokes equations \eqref{L-13} with $\H_\psi$ replaced by  $\R^d$.
We then apply the resolvent estimates \eqref{R-est-1}  in $\R^d$.
The rest of the proof is the same as that of Theorem \ref{thm-L}.
We point out  that no boundary integral appears in the estimates.
\end{proof}

%%%%%%%%%%%%%%%%%%

\section{Estimates for the velocity}\label{section-V}

Recall that $B_r =B(0, r)$ and $D_r=B(0, r) \cap \H_\psi$, where $\H_\psi$ is given by \eqref{H}.

\begin{lemma}\label{lemma-V1}
Suppose $u\in W^{1, q}(B_r)$ for some $q>d$. Then
\begin{equation}\label{V1-1}
\| u \|_{L^\infty(B_r)}
\le C r^{-\frac{d}{q}}
\left\{ \| u \|_{L^q(B_r)}
+  r \| \nabla u \|_{L^q(B_r)} \right\},
\end{equation}
where $C$ depends only on $d$ and $q$.
\end{lemma}

\begin{proof}
The lemma is well known and follows from a Sobolev inequality in $B_1$ by rescaling.
\end{proof}

\begin{lemma}\label{lemma-V2}
Let $u \in W^{1, q}(D_{2r})$ for some $q>d$.
Suppose $\|\nabla^\prime \psi  \|_\infty\le M$ for some $M>0$.
Then
\begin{equation}\label{V2-1}
\| u \|_{L^\infty(D_r)}
\le C r^{-\frac{d}{q}}
\left\{ \| u \|_{L^q(D_{2r})}
+  r \| \nabla u \|_{L^q(D_{2r} )} \right\},
\end{equation}
where $C$ depends only on $d$, $q$ and $M$.
\end{lemma}

\begin{proof}
This lemma is also more or less well known.
See \cite{AGH-2015} for a proof in the case of uniform $C^1$ domains by a compactness argument.
Here we sketch a direct argument for the case of Lipschitz domains.
By dilation we may assume $r=1$.
Moreover, in the case of the flat boundary where  $\psi=0$, the estimate follows from the inequality,
\begin{equation}
\Big|u(x) -\fint_{B^+_1} u\Big |
\le C \int_{B^+_1} \frac{|\nabla u(y)|}{|x-y|^{d-1}} \, dy
\end{equation}
for any $x\in B_1^+$, where $B_1^+=\{(x^\prime, x_d) \in B_1: x_d>0 \}$.
In the general case, by flattening the boundary, one may show that 
$$
\| u \|_{L^\infty(D_r)}
\le C r^{-\frac{d}{q}}
\left\{ \| u \|_{L^q(D_{C_0r})}
+  r \| \nabla u \|_{L^q(D_{C_0r} )} \right\},
$$
for some $C_0= C_0(M)>1$.
This, together with \eqref{V1-1}, yields \eqref{V2-1} by a covering argument.
\end{proof}

\begin{lemma}\label{lemma-V3}
Let $(u, p)$ be the same as in Theorem \ref{thm-L2}.
Suppose  $q>d$ and $r=R |\lambda|^{-1/2}$ for some $R\ge 1$.
Then
\begin{equation}\label{V3-00}
\aligned
|\lambda| \| u \|_{L^\infty(B_r)}
& \le C \left\{
\| F \|_{L^\infty (B_{Nr})}  R N^{\frac{d}{q}}
+ |\lambda| \| u \|_{L^\infty(B_{Nr})} N^{\frac{d}{q}-1}\right\}\\
& +C |\lambda|^{\frac12} \left(\fint_{B_{Nr}} |\nabla u|^q \right)^{1/q}  N^{\frac{d}{q}-1} 
+ C |\lambda|^{\frac12} \inf_{\beta\in \C} \left(\fint_{B_{Nr}} |p-\beta |^q \right)^{1/q}  N^{\frac{d}{q}-1},
\endaligned
\end{equation}
where $C$ depends only on $d$, $\theta$ and  $q$.
\end{lemma}

\begin{proof}
It follows from \eqref{V1-1} that 
\begin{equation}\label{V3-1}
\aligned
|\lambda| \| u \|_{L^\infty(B_r)}
 & \le C r^{-\frac{d}{q}}
\left\{ |\lambda | \| u \|_{L^q(B_r)}
+|\lambda|^{\frac12} R  \| \nabla u\|_{L^q(B_r)} \right\}\\
& \le C  R r^{-\frac{d}{q}}
\left\{ |\lambda | \| u \|_{L^q(B_r)}
+|\lambda|^{\frac12}   \| \nabla u\|_{L^q(B_r)} \right\},
\endaligned
\end{equation}
where we have used the assumptions  that $q>d$ and $r=R  |\lambda|^{-1/2}$ for some $R\ge 1$.
This, together with \eqref{L-22}, gives
\begin{equation}\label{V3-2}
\aligned
|\lambda| \| u \|_{L^\infty(B_r)}
 & \le C R  \Bigg\{
\| F \|_{L^\infty(B_{Nr})}  N^{\frac{d}{q}}
+  \| u \|_{L^\infty(B_{Nr})} \left( N^{\frac{d}{q}-2} r^{-2} + N^{\frac{d}{q} -1} r^{-1} |\lambda|^{\frac12}\right)\\
& \qquad \qquad+  \left( \fint_{B_{Nr}} |\nabla u|^q \right)^{1/q}  N^{\frac{d}{q}-1} r^{-1} 
+ \left(\fint_{B_{Nr}} |p|^q \right)^{1/q}  N^{\frac{d}{q}-1}  r^{-1} \Bigg\}\\
& \le C   \Bigg\{
\| F \|_{L^\infty(B_{Nr})}   R N^{\frac{d}{q}}
+  |\lambda| \| u \|_{L^\infty(B_{Nr})}  N^{\frac{d}{q}-1}\\
& \qquad \qquad+  |\lambda|^{\frac12}  \left(\fint_{B_{Nr}} |\nabla u|^q \right)^{1/q}  N^{\frac{d}{q}-1} 
+  |\lambda|^{\frac12} \left(\fint_{B_{Nr}} |p|^q \right)^{1/q}  N^{\frac{d}{q}-1}   \Bigg\}.
\endaligned
\end{equation}
The estimate \eqref{V3-00} follows from \eqref{V3-2}, 
as $(u, p-\beta)$ is also a solution for any $\beta \in \C$.
\end{proof}

The next lemma is the boundary version of Lemma \ref{lemma-V3}.

\begin{lemma}\label{lemma-V4}
Let $(u, p)$ be the same as in Theorem \ref{thm-L}.
Suppose  $q>d$ and $r=R |\lambda|^{-1/2}$ for some $R\ge 1$.
Then
\begin{equation}\label{V3-0}
\aligned
|\lambda| \| u \|_{L^\infty(D_r)}
& \le C \left\{
\| F \|_{L^\infty(D_{Nr})} R  N^{\frac{d}{q}}
+ |\lambda| \| u \|_{L^\infty(D_{Nr})} N^{\frac{d}{q}-1}\right\}\\
& \qquad
+ C |\lambda|^{\frac12} N^{\frac{d}{q}-1} \left\{ \left(\fint_{D_{Nr}}  |\nabla u|^q \right)^{1/q} 
+ \left(\fint_{I_{Nr}} |\nabla u|^q \right)^{1/q} \right\}\\
& \qquad
+ C |\lambda|^{\frac12} N^{\frac{d}{q}-1}  \inf_{\beta\in \C}
\left\{ 
 \left(\fint_{D_{Nr}} |p-\beta |^q \right)^{1/q}  
 + \left(\fint_{I_{Nr}} |p-\beta |^q \right)^{1/q}  \right\},
\endaligned
\end{equation}
where $C$ depends only on $d$, $\theta$ and  $q$.
\end{lemma}

\begin{proof}
The proof is similar to that of Lemma \ref{lemma-V3}, using Lemma  \ref{lemma-V2} and Theorem \ref{thm-L}.
We omit the details.
\end{proof}

Let $\Omega$ be a bounded  domain with $C^1 $ boundary.
Let $c_0=c_0(d, q,  \theta)\in (0, 1)$ be given by Theorem \ref{thm-L}.
Then there  exists $r_0>0$ with the property that  for any $x_0\in \partial\Omega$,
there exists a Cartesian coordinate system with origin $x_0$, obtained by translation and rotation, 
such that
\begin{equation}\label{coord}
\aligned
B(x_0, 16r_0)\cap \Omega
 & = B(x_0, 16r_0)
\cap \left\{ (x^\prime, x_d)\in \R^d:
x_d > \psi (x^\prime)\right \},\\
B(x_0, 16r_0)\cap \partial \Omega
 & = B(x_0, 16r_0)
\cap \left\{ (x^\prime, x_d)\in \R^d:
x_d = \psi (x^\prime)\right \},\\
\endaligned
\end{equation}
where $\psi: \R^{d-1} \to \R$ is a $C^1$ function and  $\|\nabla^\prime \psi\|_\infty\le  c_0$.

We need some regularity estimates for the Stokes equations (with $\lambda=0$) in $C^1$ domains.

\begin{lemma}\label{lemma-C-1}
Let $\Omega$ be a bounded $C^1$ domain and $2\le q< \infty$.
Let $(u, p)\in W^{1, 2}_0(\Omega; \C^d) \times L^2(\Omega; \C)$ be a weak solution of 
\begin{equation} \label{C-1-0}
\left\{
\aligned
-\Delta u + \nabla p & =F & \quad & \text{ in } \Omega,\\
\text{\rm div} (u) & = 0 & \quad & \text{ in } \Omega,\\
u& =0 & \quad & \text{ on } \partial \Omega,
\endaligned
\right.
\end{equation}
where $F\in L^\infty(\Omega; \C^d)$.
Then, if $B(x, 2r) \subset \Omega$,
\begin{equation}\label{C-1-1}
\left(\fint_{B(x, r)} |\nabla u|^q \right)^{1/q}  \le C\left\{ r^{-1} \| u \|_{L^\infty(\Omega)} + r \| F \|_{L^\infty(\Omega)} \right\},
\end{equation}
and if $x_0\in \partial\Omega$ and $0< r<r_0$,
\begin{equation}\label{C-1-2}
\left( \fint_{B(x_0, r)\cap \Omega} |\nabla u|^q \right)^{1/q}
+
\left( \fint_{B(x_0, r)\cap \partial\Omega} |\nabla u|^q \right)^{1/q}
\le C\left\{ r^{-1} \| u \|_{L^\infty(\Omega)} + r \| F \|_{L^\infty(\Omega)} \right\},
\end{equation}
 where $C$ depends only on $d$, $q$ and $\Omega$.
\end{lemma}

\begin{proof}

The estimate \eqref{C-1-1} follows from the well known interior regularity estimates for the Stokes equations \cite{Galdi}.
To see \eqref{C-1-2}, let $x_0\in \partial\Omega$ and $0< r< r_0$.
 By the $W^{1, q}$ estimate for the Stokes equations in $C^1$ domains \cite{FKV-1988, GSS-1994, Mitrea-2004} and a  localization procedure,
we have
\begin{equation}\label{C-1-2a}
\aligned
\left( \fint_{B(x_0, 2r)\cap \Omega} |\nabla u|^q \right)^{1/q}
\le &  \frac{C}{r}
\left( \fint_{B(x_0,4 r)\cap \Omega} | u|^2 \right)^{1/2}
+ C r \| F \|_{L^\infty(B(x_0, 4r) \cap \Omega)}\\
& \le C\left\{ r^{-1} \| u \|_{L^\infty(\Omega)} + r \| F \|_{L^\infty(\Omega)} \right\},
\endaligned
\end{equation}
where we have used the fact $u=0$ on $\partial\Omega$.
To bound the second term in the left-hand side of  \eqref{C-1-2},  we introduce 
$$
\left\{
\aligned
w(x)  & =\int_{\Omega\cap B(x_0, 4r) } \Phi (x-y) F(y)\, dy,\\
\pi (x) & =\int_{\Omega\cap B(x_0, 4r) }  \Pi(x-y) F(y)\, dy,
\endaligned
\right.
$$
where $(\Phi(x), \Pi (x) )$ denotes the fundamental solution for the Stokes operator in $\R^d$, with pole at the origin.
Let $v=u- w$.
Note that 
$$
-\Delta v + \nabla (p-\pi ) =0 \quad \text{ and } \quad
\text{\rm div} (v)=0
$$
in $ \Omega \cap B(x_0, 4r)$.
Since $\Omega$ is $C^1$, it follows that
\begin{equation}\label{C-1-3}
\fint_{B(x_0, r) \cap \partial\Omega}
|\nabla v|^q
\le C \fint_{B(x_0, 2r) \cap \partial\Omega}
|\nabla_{\tan} v |^q
+ C 
\fint_{B(x_0, 2r)\cap \Omega}|  \nabla v|^q,
\end{equation}
where  $\nabla_{\tan} v$ denotes the tangential gradient of $v$ on $\partial\Omega$ and $C$ depends only on $d$, $q$ and $\Omega$.
The inequality \eqref{C-1-3} follows from Theorem \ref{thm-B1}
by a localization procedure.

Since $v=-w$ on $B(x_0, 4r) \cap \partial\Omega$, it follows from \eqref{C-1-3} that
$$
\aligned
\left( \fint_{B(x_0, r) \cap \partial\Omega}
|\nabla u|^q\right)^{1/q} 
& \le  C\left(  \fint_{B(x_0, 2r) \cap \partial\Omega}
|\nabla w|^q\right)^{1/q} \\
& \quad +  C \left(\fint_{B(x_0, 2r)\cap \Omega} |\nabla u|^q\right)^{1/q}
+ C \left(\fint_{B(x_0, 2r)\cap \Omega} |\nabla w|^q\right)^{1/q}.\\
\endaligned
$$
Finally, we note that  $|\nabla \Phi (x)|\le C |x|^{1-d}$ and  as a result, 
$$
\|\nabla w \|_{L^\infty(B(x_0, 4r))}
\le C r \| F \|_{L^\infty(\Omega)}.
$$
This, together with \eqref{C-1-2a}, gives the desired estimate for the second term in the left-hand side of 
\eqref{C-1-2}.
\end{proof}

\begin{lemma}\label{lemma-C-2}
Let $\Omega$ be a bounded $C^1$ domain and $2\le q< \infty$.
Let $(u, p)$ be a weak solution of \eqref{eq-0}, where $F\in L^\infty(\Omega; \C^d)$.
Assume that $|\lambda|> r_0^{-2}$. 
Then, if $B(x, 2r) \subset \Omega$ and $r\ge |\lambda|^{-1/2}$, 
\begin{equation}\label{C-2-0}
\left(\fint_{B(x, r)} |\nabla u|^q \right)^{1/q}
 \le C \left\{ |\lambda|^{1/2} \| u \|_{L^\infty(\Omega)} + |\lambda|^{-1/2} \| F \|_{L^\infty(\Omega)}\right\}.
 \end{equation}
 Moreover, if $x_0\in \partial\Omega$ and $ |\lambda|^{-1/2} \le r < r_0$,  then 
 \begin{equation}\label{C-2-1}
 \aligned
 &  \left( \fint_{B(x_0, r)\cap \Omega} |\nabla u|^q \right)^{1/q}
+
\left( \fint_{B(x_0, r)\cap \partial\Omega} |\nabla u|^q \right)^{1/q}\\
& \qquad
\le C\left\{  |\lambda|^{1/2} \| u \|_{L^\infty(\Omega)} + |\lambda|^{-1/2}  \| F \|_{L^\infty(\Omega)} \right\}, 
\endaligned
\end{equation}
 where $C$ depends only on $d$, $q$ and $\Omega$.
\end{lemma}

\begin{proof}

We apply the estimates in Lemma \ref{lemma-C-1} with $r= c_0 |\lambda|^{-1/2}$.
This yields the estimates \eqref{C-2-0}-\eqref{C-2-1} for the case $r=c_0 |\lambda|^{-1/2}$. 
The general case follows  by covering $B(x, r)$ with a finite number of balls $\{ B(x_\ell, c_0 |\lambda|^{-1/2})\} $, where $x_\ell \in B(x, r)$, 
with a finite overlap.
\end{proof}

\begin{remark}\label{Lip-B-re}
If $\Omega$ is a bounded Lipschitz domain in $\R^d$, $d\ge 2$, then
the estimate \eqref{C-2-1} holds for $ 2\le q< 2+\e$, where $\e>0$ depends on $\Omega$.
This follows from the same argument as that used in the proof of Lemma \ref{lemma-C-1},
using the regularity results obtained in \cite{FKV-1988} for the Stokes equations in  Lipschitz domains.
See Theorem \ref{thm-B1}.
\end{remark}

\begin{remark}\label{re-smooth}
Let $\Omega$ be a bounded $C^{1, \alpha}$ domain for some $\alpha>0$.
We may replace the $L^q$ average in \eqref{C-1-1}-\eqref{C-1-2} by the $L^\infty$ norm, using  \cite[Theorem 2.8]{GM-1982}
(there is a typo in part b) of Theorem 2.8 in \cite{GM-1982}; $C^{1, \alpha^\prime}(\bar{B}_1^+)$ should be
$C^{0, \alpha^\prime}(\bar{B}_1^+)$, which gives the Lipschitz estimates for $C^{1, \alpha}$ domains).
As a result, if $(u, p)$ is a weak solution of \eqref{eq-0} with $F\in L^\infty(\Omega; \C^d)$, then
\begin{equation}\label{Lip-est}
\|\nabla u \|_{L^\infty(\Omega)}
\le C \left\{ ( |\lambda|+1)^{1/2} \| u \|_{L^\infty(\Omega)}
+ (|\lambda|+1)^{-1/2} \| F \|_{L^\infty(\Omega)}\right \}.
\end{equation}
\end{remark}

We are  now ready to state and prove the main result of this section.

\begin{thm}\label{thm-V}
Let  $\Omega$ be a bounded $C^{1}$ domain.
Let $\lambda\in \Sigma_\theta$ and $d<q<\infty$.
Let $(u, p) \in W_0^{1, q} (\Omega; \C^d) \times L^q(\Omega; \C) $ be a weak solution of
\begin{equation}\label{V5-0}
\left\{
\aligned
-\Delta u +\nabla p +\lambda u & = F &\quad & \text{ in } \Omega,\\
\text{\rm div} (u) & =0 & \quad & \text{ in } \Omega,\\
u& = 0& \quad & \text{ on } \partial\Omega,
\endaligned
\right.
\end{equation}
where  $F\in L^\infty(\Omega; \C^d)$. Let $t=|\lambda|^{-1/2}$ and $N\ge 2$.
\begin{equation}\label{V5-00}
\aligned
  |\lambda| \| u \|_{L^\infty(\Omega)}
 & \le C  N^{4d+8} \| F \|_{L^\infty(\Omega)}  
+ C N^{\frac{d}{q}-1}
 |\lambda| \| u \|_{L^\infty(\Omega)}\\
& + C N^{\frac{d}{q}-1}
|\lambda|^{\frac12}
\sup_{\substack{ t< r< r_0\\ B(x, 2r ) \subset \Omega}}
\inf_{\substack {\beta \in \C }}
\left(\fint_{B(x, r ) } | p -\beta|^q \right)^{1/q}\\
&+ C N^{\frac{d}{q}-1}
|\lambda|^{\frac12}
\sup_{\substack{ x_0\in \partial \Omega \\ t< r< r_0 } }
\inf_{\substack {\beta \in \C }}
\left\{
\left(\fint_{B(x_0, r )\cap \Omega } | p -\beta|^q \right)^{1/q}
+ \left(\fint_{B(x_0, r  )\cap \partial\Omega}
|p -\beta|^q \right)^{1/q} \right\},
\endaligned
\end{equation}
where $C$ depends only on $d$, $q$, $\theta$ and $\Omega$.
\end{thm}

\begin{proof}
Suppose $B(x, 2r) \subset \Omega$ for some $r>0$.
By the standard interior regularity theory for the Stokes equations (with $\lambda=0$) (see e.g. \cite[Ch. IV]{Galdi}), 
 \begin{equation}\label{V5-1a}
 \left(\fint_{B(x, r)} |u|^{q_2}\right)^{1/q_2}
 \le C \left\{
 \left(\fint_{B(x, 2r)} |u|^2 \right)^{1/2}
 + r^2 \left(\fint_{B(x, 2r)} | F-\lambda u|^{q_1} \right)^{1/q_1}\right\},
 \end{equation}
 where $\frac{1}{q_2}=\frac{1}{q_1}-\frac{2}{d}$ for $d\ge 3$ and $\frac{1}{q_2}>\frac{1}{q_1} -1$ for $d=2$, and
 \begin{equation}\label{V5-1}
\|u \|_{L^\infty(B(x, r))}
\le C \left\{ 
\left(\fint_{B(x, 2r )} |u|^2\right)^{1/2}
+ r ^2\left(\fint_{B(x, 2r ) } |F-\lambda u |^q \right)^{1/q} \right\}.
\end{equation}
It follows from \eqref{V5-1a} by a bootstrapping argument that
\begin{equation}
\left(\fint_{B(x, r)} |u|^q \right)^{1/q}
\le C (1+ r^2 |\lambda|)^d 
\left\{ \left(\fint_{B(x, 2 r)} |u|^2 \right)^{1/2}
+ r^2 \| F \|_{L^\infty(\Omega)} \right\}
\end{equation}
if $B(x, 2 r) \subset \Omega$.
This, together with \eqref{V5-1}, leads to
\begin{equation}\label{V5-1b}
\|  u \|_{L^\infty (B(x, r))}
\le C (1+ r^2 |\lambda|)^{d+1}
\left\{  \left(\fint_{B(x, 2 r)} |u|^2 \right)^{1/2}
+ r^2 \|F \|_{L^\infty(\Omega)} \right\}.
\end{equation}
Under the assumption that $\Omega$ is $C^{1}$, we  have the boundary estimate \cite{GM-1982},
\begin{equation}\label{V5-2}
\| u \|_{L^\infty(B(x_0, r)\cap \Omega)}
\le C \left\{ 
\left(\fint_{B(x_0, 2r )\cap \Omega} |u|^2\right)^{1/2}
+ r^2 \left(\fint_{B(x_0, 2r )\cap \Omega } |F-\lambda u |^q \right)^{1/q} \right\},
\end{equation}
as well as a boundary version of \eqref{V5-1a}, 
where $x_0\in \partial\Omega$, $0< r< r_0$, and
$C$ depends only on $d$, $q$ and $\Omega$.
As a result,  by a similar bootstrapping argument as in the interior case,
we also obtain 
\begin{equation}\label{V5-2b}
\| u \|_{L^\infty (B(x_0, r)\cap \Omega)}
\le C (1+ r^2 |\lambda|)^{d+1}
\left\{  \left(\fint_{B(x_0, 2 r)\cap \Omega } |u|^2 \right)^{1/2}
+ r^2 \|F \|_{L^\infty(\Omega)} \right\}.
\end{equation}
It follows from \eqref{V5-1b} and \eqref{V5-2b}  as well as the $L^2$ resolvent estimate $(|\lambda| +1) \| u \|_{L^2(\Omega)}
\le C \| F \|_{L^2(\Omega)} $ that
\begin{equation}\label{V5-3}
\| u \|_{L^\infty (\Omega)}
\le C (1+|\lambda|)^{d+1}  \| F \|_{L^\infty(\Omega)}.
\end{equation}
Consequently, it suffices to prove \eqref{V5-00} for the case $|\lambda|>16  r_0^{-2}  N^4$.

Next, let $r=2N|\lambda|^{-1/2}$.
Since $|\lambda|> 16r_0^{-2} N^4$, we have $2Nr< r_0$.
It follows from Lemma \ref{lemma-V4} by translation and rotation that for any $x_0 \in \partial\Omega$,
\begin{equation}\label{V5-5}
\aligned
& |\lambda| \| u \|_{L^\infty(B(x_0, r) \cap \Omega)}\\
& \le C \left\{
\| F \|_{L^\infty(\Omega)}   N^{\frac{d}{q} +1 } 
 +   |\lambda| \| u \|_{L^\infty(\Omega)} N^{\frac{d}{q}-1}\right\}\\
&\qquad
+|\lambda|^{\frac12}  N^{\frac{d}{q}-1}  \left\{
 \left(\fint_{B(x_0, Nr)\cap \Omega} |\nabla u|^q \right)^{1/q}
 + \left(\fint_{B(x_0, Nr) \cap \partial \Omega } |\nabla u|^q \right)^{1/q} \right\} \\
& \qquad
+ C |\lambda|^{\frac12} N^{\frac{d}{q}-1}  \inf_{\beta\in \C}
\left\{ 
 \left(\fint_{B(x_0, Nr)\cap \Omega} |p-\beta |^q \right)^{1/q}  
 + \left(\fint_{B(x_0, Nr) \cap \partial\Omega} |p-\beta |^q \right)^{1/q}  \right\}.
\endaligned
\end{equation}
This, together with \eqref{C-2-1}, gives
\begin{equation}\label{V5-6}
\aligned
& |\lambda| \| u \|_{L^\infty (\Omega_r )}\\
& \le C \left\{
\| F \|_{L^\infty(\Omega)}   N^{\frac{d}{q} +1 }
+|\lambda| \| u \|_{L^\infty(\Omega)} N^{\frac{d}{q}-1}\right\}\\
& 
+ C |\lambda|^{\frac12} N^{\frac{d}{q}-1}  \sup_{x_0\in \partial\Omega} \inf_{\beta\in \C}
\left\{ 
 \left(\fint_{B(x_0, Nr)\cap \Omega} |p-\beta |^q \right)^{1/q}  
 + \left(\fint_{B(x_0, Nr) \cap \partial\Omega} |p-\beta |^q \right)^{1/q}  \right\},
\endaligned
\end{equation}
where $\Omega_r = \{ x\in \Omega: \text{\rm dist}(x, \partial\Omega)< r \}$ and $r=2N|\lambda|^{-1/2}$.

Finally, let $x \in \Omega\setminus \Omega_r$. Then $B(x, r) = B(x, 2N |\lambda|^{-1/2} )\subset \Omega$.
This allows us to apply Lemma \ref{lemma-V3}  with $R=1$ to obtain 
\begin{equation}\label{V5-7}
\aligned
|\lambda| \| u \|_{L^\infty(B(x, |\lambda|^{-1/2} ))}
 & \le C \left\{ 
\| F \|_{L^\infty(\Omega)}  N^{\frac{d}{q}}
+ |\lambda| \| u \|_{L^\infty(\Omega)} N^{\frac{d}{q}-1}\right\}\\
&\qquad
+CN^{\frac{d}{q}-1} |\lambda|^{\frac12}
\inf_{\beta\in \C}
\left(\fint_{B(x, N|\lambda|^{-1/2})}
|p -\beta|^q \right)^{1/q},
\endaligned
\end{equation}
where we also used \eqref{C-2-0}.
The desired estimate \eqref{V5-00} follows readily from \eqref{V5-6} and   \eqref{V5-7}.
\end{proof}

\begin{remark}\label{re-V}
Let $\Omega$ be an exterior domain with $C^{1}$ boundary.
Let $(u, p)\in W^{1, q}_{\loc} (\Omega; \C^d) \times L^q_{\loc} (\Omega; \C)$ be a weak solution 
of \eqref{V5-0}, where $F\in C_{0, \sigma}^\infty (\Omega)$ and $d<q< \infty$.
We point out   that the argument for the case $|\lambda|> 16r_0^{-2} N^4$ does not 
use the boundedness of $\Omega$. In fact, 
the boundedness of $\Omega$ is only used to derive \eqref{V5-3}.
As a result, the estimate \eqref{V5-00} continues to hold for the exterior $C^{1}$  domain under the 
additional assumption that $|\lambda|> 16 r_0^{-2} N^4$.
\end{remark}

%%%%%%%%%%%%%%

\section{A Neumann problem in a bounded Lipschitz domain}\label{section-N}

In order to bound the terms involving the pressure $p$ in \eqref{V5-00}, 
we consider  the Neumann problem for Laplace's equation,
\begin{equation}\label{NP0}
\left\{
\aligned
 \Delta \phi & =0 & \quad & \text{ in } \Omega,\\
 \frac{\partial \phi}{\partial n} &= (n_i \partial_j -n_j \partial_i) g_{ij} & \quad & \text{ on } \partial\Omega,
\endaligned
\right.
\end{equation}
where the repeated indices $i, j$ are summed from $1$ to $d$ and $g_{ij} \in C^1(\partial\Omega)$. 
Throughout this section we assume that  $\Omega$ is a bounded Lipschitz domain in $\R^d$, $d\ge 2$.

Let $g=(g_{ij})$ and  $I(x_0, r)= B(x_0, r) \cap \partial\Omega$.
For $0< r<   r_0$,  define
\begin{equation}\label{M-t}
M_q (g)  (r) =\sup_{\substack{x_0 \in \partial\Omega \\ r\le  s< r_0}}
\left(\fint_{I(x_0, s)} |g|^q \right)^{1/q}.
\end{equation}
It is not hard to see that $M_q(g) (r) \approx M_q(g) (2r)$.

\begin{lemma}\label{lemma-P1}
Let $\phi  \in C^1(\overline{\Omega})$ be a weak solution of  \eqref{NP0}.
Let $2\le q< \infty$.
Then, for $0< r< r_0$, 
\begin{equation}\label{P1-0}
\sup_{\substack{x_0 \in \partial\Omega }}
\left(\fint_{I(x_0, r)}
|\phi  -\fint_{I(x_0, r)} \phi |^q \right)^{1/q}
\le C M_q (g) (r),  
\end{equation}
where $C$ depends only on $d$, $q$ and the Lipschitz character of $\Omega$.
\end{lemma}

\begin{proof}
Fix $x_0 \in \partial \Omega$ and $0< r< r_0$.
Let $v\in H^1(\Omega) $ be a weak solution of the Neumann problem,
\begin{equation}\label{P1-1}
\Delta v=0 \quad \text{ in } \Omega  \quad \text{ and } \quad \frac{\partial v}{\partial n } =\eta \quad \text{ on } \partial \Omega,
\end{equation}
where $\eta  \in L^2 (\partial\Omega) $, supp$(\eta) \subset I(x_0, r)$  and $\int_{\partial\Omega} \eta =0$.
Since $1< q^\prime\le 2$, 
under the assumption that $\Omega$ is Lipschitz,  the estimate
$\|(\nabla v)^* \|_{L^{q^\prime} (\partial\Omega)} \le C \| \eta \|_{L^{q^\prime} (\partial\Omega)}$ holds and $\nabla v$ has non-tangential limits a.e. on $\partial\Omega$, where
$(\nabla v)^*$ denotes the non-tangential maximal function of $\nabla v$.
Moreover, assuming $\int_{\partial \Omega} v =0$, we have 
\begin{equation}\label{NF-1}
v(x) =\int_{\partial\Omega} \left( N(x, y)-N(x, x_0) \right)  \eta (y) \, dy
\end{equation}
for any $x\in \Omega$, where $N(x, y)$ denotes the Neumann function for $-\Delta$ in $\Omega$.
Using the estimate,
\begin{equation} \label{NF}
|N(x, y) -N(x, x_0)| \le \frac{Cr^\sigma}{|x-y|^{d-2 +\sigma}}
\end{equation}
for $x\in \Omega \setminus B(x_0, 2r)$ and $y\in I(x_0, r)$, where $\sigma>0$ depends on $\Omega$, we obtain
\begin{equation}\label{BF-2}
|v(x) | \le \frac{Cr^\sigma}{|x-x_0|^{d-2 +\sigma}} \|\eta\|_{L^{q^\prime}(\partial\Omega)} r^{\frac{d-1}{q}}.
\end{equation}
We refer the reader to \cite{kenig-book} for proofs of these results.

Note that 
\begin{equation}\label{P1-2}
\aligned
\int_{I (x_0, r)} \Big (\phi -\fint_{I(x_0, r)} \phi \Big)  \eta &=
\int_{\partial\Omega} \phi \cdot \eta
  =\int_{\partial \Omega} \phi \frac{\partial v}{\partial n}
=\int_{\partial \Omega}  v \frac{\partial \phi }{\partial n }\\
&=\int_{\partial \Omega}  v (n_i \partial_j -n_j \partial_i) g_{ij} 
=-\int_{\partial\Omega}  g_{ij} (n_i \partial_j -n_j \partial_i) v.
\endaligned
\end{equation}
Hence, 
\begin{equation}\label{P1-2a}
\aligned
\big| \int_{I (x_0, r)} \Big (\phi -\fint_{I(x_0, r)} \phi \Big)  \eta\big| 
 & \le C \int_{\partial\Omega} |\nabla v| |g |
  = C \sum_{j=0}^{J+1} A_j, 
 \endaligned
\end{equation}
where $2^J r \approx r_0$, and 
\begin{equation}\label{P1-2b}
\aligned
A_0  & = \int_{I(x_0, 4r)} |\nabla v | |g|, \\
A_j & =\int_{I(x_0, 2^{j+2} r)\setminus I(x_0, 2^{j+1}  r)} |\nabla v | | g| \quad \text{ for } j=1,2, \dots, J,\\
A_{J+1} & =\int_{\partial\Omega\setminus  I(x_0, 2^{J+2}  r)} |\nabla v | | g|.
\endaligned
\end{equation}
We now proceed to bound each term in \eqref{P1-2b}.
First, observe  that by H\"older's inequality, 
\begin{equation}\label{P1-2c}
\aligned
A_0  & \le C \| \nabla v \|_{L^{q^\prime}(\partial\Omega)}  \left(\fint_{I(x_0, 4r)} |g |^q \right)^{1/q} r^{\frac{d-1}{q}}\\
  & \le C \| \eta  \|_{L^{q^\prime}(\partial\Omega)}  \left(\fint_{I(x_0, 4r)} |g |^q \right)^{1/q} r^{\frac{d-1}{q}}\\
  &  \le C \| \eta \|_{L^{q^\prime} (\partial\Omega)}  M_q(g) (r) \,   r^{\frac{d-1}{q}},
    \endaligned
\end{equation}
where we have used the notation \eqref{M-t}. 

Next,  note that if   $\frac{\partial v}{\partial n} =0$ on $I (y_0, 2R)$ for some $y_0\in \partial\Omega$ and
$r< R< r_0$, then
\begin{equation}\label{P1-2d}
\aligned
\left( \fint_{I(y_0, R)} |\nabla v|^{q^\prime}\right)^{1/q^\prime}
 & \le \left( \fint_{I(y_0, R)} |\nabla v|^{2}\right)^{1/2}
  \le \frac{C}{R}\fint_{B(y_0, 2R)\cap \Omega } |v | \\
 & \le C \left (\frac{r}{R} \right)^\sigma
  R^{1-d} \| \eta\|_{L^{q^\prime} (\partial\Omega)} r^{\frac{d-1}{q}},
  \endaligned
\end{equation}
where we have used the $L^2$ estimates for the Neumann problem for the second inequality \cite{kenig-book}
and  \eqref{BF-2} for the last.
This, together with H\"older's inequality,  allows us to deduce that
$$
A_j \le C 2^{-j \sigma}   \| \eta \|_{L^{q^\prime} (\partial\Omega)}M_q (g) (r)  \ r^{\frac{d-1}{q}}
$$
for $j=1, 2, \dots, J$.
The same argument also yields $A_{J+1} \le C  \| \eta \|_{L^{q^\prime} (\partial\Omega)} M_q (g) (r) r^{\frac{d-1}{q}}$.
By summation we obtain 
$$
\big| \int_{I (x_0, r)} \Big (\phi -\fint_{I(x_0, r)} \phi \Big)  \eta\big| 
\le 
C  \| \eta \|_{L^{q^\prime} (\partial\Omega)}M_q (g) (r) \, r^{\frac{d-1}{q}}.
$$
By duality this implies that
$$
\left(\fint_{I(x_0, r)} |\phi -\fint_{I(x_0, r) } \phi|^q \right)^{1/q}
\le C M_q(g) (r)
$$
for any $x_0\in \partial\Omega$ and $0<r< r_0$.
\end{proof}

\begin{remark}\label{re-BMO}
It follows from \eqref{P1-0}  that
\begin{equation}
\|\phi \|_{B\!M\!O(\partial\Omega)}
\le C \| g \|_{L^\infty(\partial\Omega)},
\end{equation}
where $C$ depends  on $d$ and $\Omega$.
However, this estimate is not strong enough for our approach to resolvent estimates in $C^1$ domains.
\end{remark}

Recall that  $\delta (x)=  \text{\rm dist}(x, \partial\Omega)$.

\begin{lemma}\label{lemma-N1}
Let  $2\le q< \infty$.
Let $\phi$ be the same as in Lemma \ref{lemma-P1}.
Then, for $0< r< r_0$,
\begin{equation}\label{N1-0}
\sup_{x_0 \in \partial\Omega} 
\left(\fint_{B(x_0, r) \cap \Omega} |\phi -\fint_{I(x_0, r)} \phi |^q\right)^{1/q}
\le C M_q(g) (r).
\end{equation}
Moreover, for any $x\in \Omega$,
\begin{equation}\label{N1-1}
 \delta(x) |\nabla \phi (x) | 
\le C M_q(g) (\delta (x) ).
\end{equation}
 The constants  $C$ in \eqref{N1-0}-\eqref{N1-1}  depend at most on $d$, $q$ and the Lipschitz character of $\Omega$.
 \end{lemma}

\begin{proof}
Fix $x_0 \in \partial\Omega$ and 
let $\beta=\fint_{I(x_0, r)} \phi $.
Write $\phi-\beta =\phi_1 +\phi_2$, where $\phi_1$ and $\phi_2$ are bounded harmonic functions in $\Omega$ and
$\phi_1 = (\phi - \beta)\chi_{I(x_0, 4r) }$ on $\partial\Omega$.
Note that 
$$
\aligned
\left(\fint_{B(x_0, r)\cap \Omega} 
|\phi_1|^q \right)^{1/q}
 & \le C r^{\frac{1-d}{q}} \| (\phi_1)^* \|_{L^q(\partial\Omega)}
\le C  r^{\frac{1-d}{q}} \| \phi_1 \|_{L^q(\partial\Omega)}\\
&\le C r^{\frac{1-d}{q}} \| \phi-\beta \|_{L^q(I(x_0, 4r))}
\le C  M_q(g) (r),
\endaligned
$$
where $(\phi_1)^*$ denotes the non-tangential  maximal function of $\phi_1$ and we have used Lemma \ref{lemma-P1} for 
the last step.
We point out  that in a bounded Lipschitz domain $\Omega$,
the  estimate $\| (\phi_1)^* \|_{L^q(\partial\Omega)}
\le C \| \phi_1\|_{L^q(\partial\Omega)}$ holds for any $q\ge 2$ \cite{kenig-book}. 

Next, we will show that for any $x\in B(x_0, r) \cap \Omega$, 
\begin{equation}\label{P2-2}
|\phi_2 (x)| \le C  M_q(g) (r).
\end{equation}
To this end, we represent $\phi_2$ by a Poisson integral,
$$
\phi_2 (x) =\int_{\partial \Omega \setminus B(x_0, 4r)} P(x, y) (\phi (y) -\beta) dy.
$$
Let $G(x, y) $ denote the Green function for $-\Delta$ in $\Omega$.
Note that
$$
\aligned
|\phi_2 (x)|
&\le  \int_{\partial \Omega \setminus B(x_0, 4r)}  |\nabla_y G(x, y)|  |\phi (y) -\beta  | dy\\
 &\le \sum_{j=1}^J \int_{I(x_0, 2^{j+2} r)\setminus I(x_0, 2^{j+1}  r) }
|\nabla_y G(x, y) | |\phi -\beta|\, dy \\
&\qquad\qquad
+\int_{\partial \Omega \setminus I (x_0, 2^{J+2} r)} 
|\nabla_y G(x, y)| |\phi (y) -\beta|\, dy\\
&=\sum_{j=1}^J A_j + A_{J+1}, 
\endaligned
$$
where $2^J r \approx r_0$.
Using the estimate
\begin{equation}\label{P2-3}
\int_{I (x_0, 2R) )\setminus I(x_0, R)}
|\nabla_y G(x, y)|^2\, dy
\le \frac{C}{R^3}
\int_{[B(x_0, 3R)\setminus B(x_0, R/2)]\cap \Omega} 
|G(x, y)|^2\, dy
\end{equation}
for $R\ge 4r$ and $x\in B(x_0, r)\cap \Omega$, 
and
\begin{equation}
|G(x, y) |\le \frac{C r^\sigma}{|x-y|^{d-2 +\sigma}}
\end{equation}
for $|x-y|\ge r$,  where $\sigma>0$ depends on $\Omega$, we deduce that
\begin{equation}\label{P2-3a}
\aligned
\int_{I(x_0, 2R)\setminus I(x_0, R)}
|\nabla_y G(x, y) | |\phi(y)-\beta|\, dy 
 & \le C \left(\frac{r}{R}\right)^\sigma
\left(\fint_{I(x_0, 2R) } |\phi -\beta|^2 \right)^{1/2}\\
& \le  C \left(\frac{r}{R}\right)^\sigma
\ln \left(\frac{R}{r}\right)  M_q(g)(r).
\endaligned
\end{equation}
For the second inequality in \eqref{P2-3a}, we have used \eqref{P1-0} as well as  the observation that
\begin{equation}\label{BMO-1}
\Big| \fint_{I(x_0, R)} \phi -\fint_{I(x_0, r)}  \phi\Big  |
\le C \ln \left(\frac{R}{r} \right) 
\sup_{r\le  t \le  R}
 \fint_{I(x_0, t)} \big|\phi 
-\fint_{I(x_0, t)} \phi \big|.
\end{equation}
It follows from \eqref{P2-3a}  that
$$
A_j \le C  j 2^{-j \sigma}  M_q(g)(r) 
$$ 
for $j=1, 2, \dots, J$. 
Similar,  we may show that
$$
A_{J+1} \le C r^\sigma  |\ln r |  M_q(g) (r).
$$
As a result, we obtain \eqref{P2-2} by summation and complete the proof of \eqref{N1-0}. 

Finally, since $\phi$ is harmonic in $\Omega$, we have
$$
\delta (x) |\nabla \phi (x)| \le C  \inf_{\beta\in \C} \fint_{B(x,  \delta(x))} |\phi-\beta|.
$$
This, together with \eqref{N1-0}, implies that  if $\delta (x) \le cr_0$, then
\begin{equation}\label{P2-00}
\delta (x) |\nabla \phi (x) | \le C M_q(g) (\delta(x)).
\end{equation}
By using the maximum principle on $\nabla \phi$, one may show that \eqref{P2-00} holds for any $x\in \Omega$.
\end{proof} 

\begin{remark}\label{re-P1}
It follows from \eqref{N1-1} that 
\begin{equation}\label{P-200}
\delta (x) |\nabla \phi (x) | \le C \| g \|_{L^\infty (\partial\Omega)} \quad \text{ for any } x\in \Omega.
\end{equation}
The estimate \eqref{P-200} was proved in \cite{KLS-2013} under the assumption that $\Omega$ is a $C^{1, \alpha} $ domain for some
$\alpha>0$, using the estimate for the Neumann function,
$$
|\nabla_x\nabla_y N(x, y) |\le  C |x-y|^{-d}
$$
for $x, y\in \Omega$.
Independently, the estimate \eqref{P-200} was proved for bounded and exterior domains with  $C^3$ boundaries  in the study of the $L^\infty$ Stokes semigroup by
K. Abe and Y. Giga in  \cite{AG-2013, AG-2014}. 
\end{remark}

\begin{thm}\label{thm-N1}
Let $\Omega$ be a bounded Lipschitz domain in $\R^d$, $d\ge 2$ and $2\le q< \infty$.
Let $\phi \in C^1(\overline{\Omega})$  be a solution of \eqref{NP0}.
Then, 
\begin{equation}\label{P3-0}
\sup_{B(x, 2r)\subset \Omega}
\left(\fint_{B(x, r)} |\phi - \fint_{B(x, r)} \phi |^q \right)^{1/q}
\le C M_q (g) (r) ,
\end{equation}
and for $x_0\in \partial\Omega$ and $0<r< r_0$,
\begin{equation}\label{P3-00}
\left(\fint_{B(x_0, r) \cap \partial\Omega}
|\phi -\fint_{B(x_0, r) \cap \partial\Omega} \phi |^q \right)^{1/q}
+
\left(\fint_{B(x_0, r) \cap \Omega}
|\phi -\fint_{B(x_0, r) \cap\partial \Omega} \phi  |^q \right)^{1/q}
\le C  M_q (g) (r), 
\end{equation}
where  $C$ 
depends only on $d$, $q$,  and the Lipschitz character of $\Omega$.
\end{thm}

\begin{proof}
The estimate \eqref{P3-0} follows from \eqref{N1-1} by using a Poincar\'e inequality,
while \eqref{P3-00} is contained in Lemmas \ref{lemma-P1} and \ref{lemma-N1}.
\end{proof}

\begin{remark}\label{re-N}
Let $\phi \in C^1(\overline{\Omega})$ be a weak  solution of the Neumann problem,
\begin{equation}\label{NP1}
\left\{
\aligned
 \Delta \phi & =0 & \quad & \text{ in } \Omega,\\
 \frac{\partial \phi}{\partial n} &= (n_i \partial_j -n_j \partial_i) g_{ij}  +h  & \quad & \text{ on } \partial\Omega,
\endaligned
\right.
\end{equation}
where $g_{ij} \in C^1(\partial\Omega)$ and $h \in C(\partial\Omega)$.
Then the estimates \eqref{P3-0}-\eqref{P3-00}
hold with $M_q(g)(r) $ replaced by $ M_q(g) (r)  + \| h\|_{L^\infty(\partial\Omega)}$.
Indeed, 
in  view of Theorem \ref{thm-N1},  we only need to consider the case $g=0$.
To this end,
we  assume $\int_{\partial\Omega} \phi=0$.
This allows us to represent $\phi$ by the Neumann function,
$$
\phi(x) =\int_{\partial \Omega} N(x, y) h(y)\, dy.
$$
Using the estimates $|N(x, y)|\le C |x-y|^{2-d}$ for $d\ge 3$ and
$|N(x, y)|\le C( 1+|\ln |x-y|| )$ for $d=2$, we deduce that
\begin{equation}\label{NP-110}
 \|\phi \|_{L^\infty(\Omega)} \le C \|h \|_{L^\infty(\partial\Omega)},
 \end{equation}
 which gives the desired estimates.
 \end{remark}

%%%%%%%%

\section{A Neumann problem in an exterior Lipschitz domain}\label{section-Ne}

Throughout this section we assume that  $\Omega$ is an exterior Lipschitz domain in $\R^d$, $d\ge 2$.
We begin with some classical results on the asymptotic behavior  of harmonic functions at infinity.

\begin{lemma}\label{lemma-Ne0}
Let $\phi$ be harmonic in $\Omega$.
Suppose that as $|x|\to \infty$,  $\phi (x) = O(1)$ if $d\ge 3$, and 
$\phi (x) = o(\log |x|)$ if $d=2$. Then
\begin{equation}\label{HA-1}
\phi(x)=
\left\{
\aligned
 &  \beta + O(|x|^{2-d})  &\quad & \text{ if } \ d\ge 3,\\
 & \beta + O(|x|^{-1}) & \quad & \text{ if } \  d=2,
\endaligned
\right.
\end{equation}
for some $\beta \in \C$ as $|x|\to \infty$.
Moreover,
\begin{equation}\label{HA-2}
|\nabla \phi (x)| =
\left\{
\aligned
& O(|x|^{1-d} ) & \quad & \text{ if } \ d\ge 3,\\
& O(|x|^{-2}) & \quad & \text{ if } \ d=2,
\endaligned
\right.
\end{equation}
as $|x|\to \infty$.
\end{lemma}

\begin{proof}
See  e.g. \cite{HFT}.
\end{proof}

\begin{remark}\label{re-AS1}
Let $\phi$ be harmonic in $\Omega$.
Suppose $\phi  \in C^1(\overline{\Omega})$  and $\phi (x)  \to 0$ as $|x|\to \infty$.
Let  $D=B(0, R)  \cap \Omega$, where $R$ is large.
Let $\Gamma^x(y)=\Gamma (x-y)$ denote the fundamental solution for $-\Delta$ in $\R^d$.
By applying Green's representation formula to $ \phi$,
$$
  \phi(x)
=
\int_{\partial D} \Gamma^x \frac{\partial}{\partial n} \phi 
-\int_{\partial D} \frac{\partial \Gamma^x}{\partial n}  \phi
$$
for $x\in D$, and then letting $R \to \infty$, 
we obtain
\begin{equation}\label{Green}
 \phi (x)
=\int_{\partial \Omega  } \Gamma^x \frac{\partial\phi }{\partial n} 
-\int_{\partial  \Omega } \frac{\partial \Gamma^x}{\partial n}  \phi
\end{equation}
for $x\in \Omega $. In the case $d=2$, using the observation that $\int_{\partial\Omega} \frac{\partial \phi}{\partial n}=0$, 
we have
\begin{equation}\label{Green-1}
 \phi (x)
=\int_{\partial \Omega  } \left( \Gamma^x- \Gamma^x (y_0)\right)   \frac{\partial\phi }{\partial n}
-
\int_{\partial  \Omega } \frac{\partial \Gamma^x}{\partial n}  \phi
\end{equation}
for any $x\in \Omega$,
where $y_0 \in \Omega^c$.
\end{remark}

Recall that $I(x_0, r) = B(x_0, r) \cap \partial\Omega$ and $M_q(g)(r)$ is defined by \eqref{M-t}.

\begin{lemma}\label{lemma-P1e}
Let $\phi  \in C^1(\overline{\Omega})$ be a solution of  \eqref{NP0} and 
$2\le q< \infty$.
Assume that $\phi (x) \to 0$ as $|x|\to \infty$.
Then, for $0< r< r_0$,
\begin{equation}\label{P1-0e}
\sup_{\substack{x_0 \in \partial \Omega}}
\left(\fint_{I(x_0, r)} | \phi - \fint_{I(x_0, r)} \phi|^q \right)^{1/q}
+ \big| \fint_{\partial\Omega} \phi \big| 
\le C  M_q(g )(r),
\end{equation}
where $C$ depends only on $d, q$ and the Lipschitz character of $\Omega$.
\end{lemma}

\begin{proof}
Since $\phi$ is harmonic in $\Omega$, it follows from the assumption 
$\phi (x)=o(1)$ as $|x|\to \infty$ that
\begin{equation}\label{Ne-1}
|\phi (x) | + |x| |\nabla \phi (x)| 
=\left\{
\aligned
 &  O(|x|^{2-d} ) & \quad & \text{ for } d\ge 3,\\
 &  O(|x|^{{-1} }) & \quad & \text{ for } d= 2,\\
\endaligned
\right.
\end{equation}
as $|x|\to \infty$.
Moreover,   if $\int_{\partial\Omega} \eta=0$, 
the Neumann problem \eqref{P1-1} possesses a solution with the properties that $(\nabla v)^* \in L^2(\partial\Omega)$ and 
\begin{equation}
|v(x)| + |x| |\nabla v(x)| = O(|x|^{1-d}) \quad \text{ as }  |x| \to \infty.
\end{equation}
See \cite{Verchota-1984} for a proof by the method of layer potentials in Lipschitz domains.
This, together with \eqref{Ne-1},  allows us  to justify the use of integration by parts in the third equality in \eqref{P1-2}.
As in the case of bounded domains,   we deduce by duality that  the first term in the left-hand side of \eqref{P1-0e}
is bounded by $C M_q(g)(r)$.

Next, let $v$ be a solution of the Neumann problem \eqref{P1-1} in $\Omega$ with $\eta=1$  and  the property that
$v (x) =O(|x|^{2-d})$ for $d\ge 3$, and $v(x)= O(|\ln |x|| ) $ for $d=2$ as $|x|\to \infty$.
It follows that $|\nabla v (x)| =O(|x|^{1-d}) $ for $d\ge 2$ as $|x|\to \infty$.
Note that
$$
\aligned
\int_{\partial\Omega} \phi 
= \int_{\partial \Omega} \phi \frac{\partial v} {\partial n}
=\int_{\partial \Omega} v \frac{\partial \phi}{\partial n}
-\int_{\partial B(0, R)} v \frac{\partial \phi}{\partial n}
+ \int_{\partial B(0, R)} \phi \frac{\partial v}{\partial n},
\endaligned
$$
where $R>1$ is so large that $\Omega^c \subset B(0, R/2)$.
This gives
\begin{equation}\label{Ne-1a}
\big| \int_{\partial\Omega} \phi \big|
\le C \|\nabla v \|_{L^2(\partial \Omega)} \| g \|_{L^2(\partial\Omega)}
+ \int_{\partial B(0, R)}
\left\{ | v| |\nabla \phi| + |\phi | |\nabla v| \right\}.
\end{equation}
Using the decay properties  at $\infty$ for $\phi$ and $v$,
we can show the last integral in \eqref{Ne-1a} converges to zero as $R \to \infty$.
As a result, we obtain 
$$
\big| \fint_{\partial\Omega} \phi \big|
\le C \| g \|_{L^2 (\partial\Omega)} \le C M_q(g) (r),
$$
which completes the proof of \eqref{P1-0e}.
\end{proof}

\begin{remark}\label{re-e1}
Let $\phi$ be the same as in Lemma \ref{lemma-P1e}.
Since $\phi (x)\to 0$ as $|x|\to \infty$, it follows that
\begin{equation}\label{e1-0}
\| (\phi)^*\|_{L^q(\partial\Omega)}
\le C  \| \phi \|_{L^q(\partial\Omega)}
\end{equation}
for $2\le q< \infty$. This, together with \eqref{P1-0e}, yields
\begin{equation}\label{e1-1}
\| (\phi)^*  \|_{L^q(\partial\Omega)}
\le C M_q(g)(r), 
\end{equation}
where $C$ depends only on $d$, $q$ and the Lipschitz character of $\Omega$.
\end{remark}

\begin{thm}\label{thm-N2}
Let $\Omega$ be an exterior Lipschitz domain in $\R^d$ and $2\le q< \infty$.
Let $\phi$ be the same as in Lemma \ref{lemma-P1e}.
Then, for $x_0\in \partial\Omega$ and $0< r< r_0$,
\begin{equation}\label{N2-0}
\left(\fint_{B(x_0, r) \cap \Omega} |\phi -\fint_{B(x_0, r)\cap \partial\Omega} \phi |^q\right)^{1/q}
\le C M_q(g)(r),
\end{equation}
where  $C$ 
depends only on $d$, $q$,  and the Lipschitz character of $\Omega$.
Moreover, for any $x\in \Omega$,
\begin{equation}\label{N2-1}
 \delta(x) |\nabla \phi (x) | 
 \le C  \left\{ M_q (g)(\delta(x)) + \| g \|_{L^q(\partial\Omega)} \right\}.
 \end{equation}
 \end{thm}

\begin{proof}

Fix $x_0\in \partial\Omega$.
Without loss of generality we may assume that  $x_0=0$ and
$$
B(0, 16r_0) \cap \Omega
=B(0, 16 r_0) \cap \H_\psi
\quad \text{ and } \quad
B(0, 16r_0) \cap \partial\Omega
= B(0, 16r_0)\cap \partial \H_\psi,
$$
where $\H_\psi$ is given by \eqref{H} and $\psi: \R^{d-1} \to \R$ is a Lipschitz function with $\|\nabla \psi\|_\infty\le M$.
Let $\beta=\fint_{B(0, r) \cap \partial \Omega} \phi$.
 Let $\phi_1$ and $\phi_2$ be harmonic functions in $\H_\psi$
such that 
$$
\phi_1 =(\phi-\beta) \chi_{B(0, 2r)\cap \partial\H_\psi }, \quad 
\phi_2 = (\phi-\beta) \chi_{\partial\H_\psi \cap(  B(0, 4r_0) \setminus B(0, 2r))}
$$
 on $\partial\H_\psi$,
and that $(\phi_1)^*, (\phi_2)^*\in L^2(\partial \H_\psi)$.
Using the same argument as in the proof of Theorem \ref{thm-N1},
we may show that
\begin{equation}\label{N2-4}
\aligned
\left(\fint_{B(0, r) \cap \Omega } |\phi_1 + \phi_2 |^q \right)^{1/q}
 & \le C  M_q(g)(r),
 \endaligned
\end{equation}
where we also have used Lemma \ref{lemma-P1e}.

Next, let $\phi_3 = \phi -\beta -(\phi_1 + \phi_2)$.
Note that $\phi_3$ is harmonic in $B(0, 4r_0)\cap \Omega$ and
$\phi_3=0$ on $B(0, 4r_0)\cap \partial\Omega$.
Since $\Omega$ is Lipschitz, it follows that $\phi_3$ is H\"older continuous in
$B(0, 3r_0)\cap \Omega$ and that 
$$
\left(\fint_{B(0, r) \cap \Omega} |\phi_3|^q\right)^{1/q}
\le C r^\sigma
\left(\fint_{B(0, 3r_0)\cap \Omega} |\phi_3 |^q \right)^{1/q}
$$
for some $\sigma>0$.
This leads to
$$
\aligned
\left(\fint_{B(0, r) \cap \Omega} |\phi_3|^q\right)^{1/q}
& \le C r^\sigma
\left\{ \| (\phi)^* \|_{L^q(\partial\Omega)}
+ |\beta|
+ \| (\phi_1 + \phi_2)^* \|_{L^q(\partial\H_\psi)} \right\}\\
& \le C r^\sigma
\left\{
M_q(g)(r) 
+  |\beta|
+ \| \phi-\beta \|_{L^q(\partial\Omega)} \right\}\\
&\le C r^\sigma (1+ |\ln r|)  M_q (g) (r)\\
& \le C  M_q (g)(r),
\endaligned
$$
where we have used  \eqref{e1-1} and \eqref{BMO-1}. 
As a result, we obtain \eqref{N2-0}.

 As in the case of bounded domains,  it follows from \eqref{N2-0} that  
 $$
 \delta(x) |\nabla \phi (x) |\le C M_q(g)(\delta(x) )
 $$
  if $\delta(x) \le cr_0$.
 Since $|\nabla \phi (x)| \to 0$ as $|x| \to \infty$, by using the maximum principle,
 we deduce that $ |\nabla \phi (x) |\le  C \| g \|_{L^q(\partial\Omega) } $ if $\delta (x) \ge cr_0$.
 Finally, choose $R_0$ large so that $\Omega^c \subset B(0, R_0-3)$.
 By applying the formulas in \eqref{Green}-\eqref{Green-1} to the harmonic functions 
 $\nabla \phi$ in $B(0, R_0)^c$, we may show that for any $x\in B(0, R_0+2)^c$, 
 $$
 |\nabla \phi (x) |\le
 \left\{
 \aligned
 &  C |x|^{2-d} \|\nabla \phi \|_{L^\infty (B(0, R_0+1)\setminus B(0, R_0-1))} & \quad & \text{ if } d\ge 3,\\
 & C |x|^{-1} \|\nabla \phi \|_{L^\infty(B(0, R_0+1)\setminus B(0, R_0-1))} & \quad & \text{ if } d=2.
 \endaligned
 \right.
 $$
 Consequently, we obtain \eqref{N2-1} for any $x\in \Omega$.
\end{proof}

\begin{remark}\label{re-EN}

Let  $\phi\in C^1(\overline{\Omega})$ be a solution of the Neumann problem \eqref{NP1},
where $g_{ij} \in C^1(\partial\Omega)$ and $h \in C(\partial\Omega)$.
Assume that $\phi (x) \to 0$ as $|x|\to \infty$.
In the case $d=2$, we also assume that $\int_{\partial\Omega} h =0$.
Then, as in the case of bounded domains,
estimates \eqref{P1-0e}, \eqref{N2-0} and \eqref{N2-1} hold
if we add the term $C \| h \|_{L^\infty(\partial\Omega)}$ in the right-hand sides.

To see this, it suffices to show that there exists a harmonic function $v$ in $\Omega$ such that
$\frac{\partial v}{\partial n} =h$ on $\partial\Omega$, $v(x) \to  0$ as $|x|\to \infty$, and
$\| v\|_{L^\infty(\Omega)}\le C \| h \|_{L^\infty(\partial\Omega)}$.
To this end, we first note that there exists a harmonic function in $\Omega$ such that $\frac{\partial v}{\partial n} =h$ on $\partial\Omega$
and $\| (\nabla v)^*\|_{L^2(\partial\Omega)} \le C \| h \|_{L^2(\partial\Omega)}$.
If $d\ge 3$, the solution satisfies  $v(x)= O(x|^{2-d})$ as $|x|\to \infty$.
In the case $d=2$, using the assumption that $\int_{\partial \Omega} h =0$, we have
$ v(x) = O(|x|^{-1})$ as $|x|\to \infty$.
We refer the reader to \cite{Verchota-1984} for a proof in Lipschitz domains.
Furthermore, we note that in \cite{Verchota-1984},  the solution was constructed using
a single layer potential with a density $f$ satisfying $ \|  f \|_{L^2(\partial\Omega)} \le C \| h \|_{L^2(\partial\Omega)}$.
In the case $d=2$, we also have $\int_{\partial\Omega} f =0$.
As a consequence, we obtain  $\| (v)^* \|_{L^2(\partial\Omega)} \le C \| h \|_{L^2(\partial\Omega)}$.

Finally, we use the local estimates for Neumann problem on Lipschitz domains \cite{kenig-book}  to obtain
$$
\aligned
\| v \|_{L^\infty (\Omega_{r_0})}
 & \le C\left\{  \| v \|_{L^2(\Omega_{2r_0})}
+  \| h \|_{L^\infty(\partial\Omega)}\right\} \\
& \le C \left\{ \| (v)^* \|_{L^2(\partial\Omega)} + \| h \|_{L^\infty(\partial\Omega)}\right \}\\
& \le C \| h \|_{L^\infty(\partial\Omega)},
\endaligned
$$
where $\Omega_r = \{ x\in \Omega: \delta (x) < r\}$.
This, together with the maximum principle, yields  the desired estimate $|v(x) |\le C \| h \|_{L^\infty(\partial\Omega)}$ for any $x\in \Omega$.
\end{remark}

%%%%%%%%%%%%%%%%%%%%

\section{Estimates for the pressure}\label{section-P1}

In this section we establish the key estimates for the pressure.
We first deal with the case of bounded domains.
Recall that $\delta(x) =\text{\rm dist}(x, \partial\Omega)$.

\begin{thm}\label{thm-P1}
Let $\lambda\in \Sigma_\theta$ and $q>2$.
Let $(u, p)$ be a weak solution of \eqref{eq-0}, where
$F\in C_{0, \sigma}^\infty(\Omega)$ and $\Omega$ is a bounded Lipschitz domain.
 Suppose $u\in C^1(\overline{\Omega} ; \C^d)$ and $p \in C(\overline{\Omega}; \C)$.
Let  $B(x, 2r)\subset \Omega$.
Then
\begin{equation}\label{P-99}
\left(
\fint_{B(x, r)} |p-\fint_{B(x, r)}  p |^q \right)^{1/q}
\le C  M_q(|\nabla u|)(\delta(x) ),
\end{equation}
where $C$ depends only on $d$, $q$ and $\Omega$.
\end{thm}

\begin{thm}\label{thm-P2}
Let $(u, p)$ be the same as in Theorem \ref{thm-P1}.
Suppose $x_0 \in \partial\Omega$ and $0< r<  r_0$.
Then
\begin{equation}\label{P-97}
\left(
\fint_{B(x_0, r)\cap \Omega} |p-\beta|^q \right)^{1/q}
+\left(\fint_{B(x_0, r)\cap \partial\Omega}
|p-\beta|^q \right)^{1/q} 
\le C M_q(|\nabla u|) (r),
\end{equation}
where $\beta=\fint_{B(x_0, r) \cap \partial\Omega} p$ and  $C$ depends only on $d$, $q$ and $\Omega$.
\end{thm}

\begin{proof}

The proofs of these two theorems rely on the a priori estimates we established in Section \ref{section-N}  for the Neumann problem \eqref{NP1}.
We begin by choosing a sequence of smooth domains $\{ \Omega_\ell\}$ with uniform Lipschitz characters such that 
$\Omega_\ell \nearrow  \Omega$.
Let $n^\ell$ denote the outward unit normal to $\partial\Omega_\ell$.
Since $F\in C_{0, \sigma}^\infty(\Omega)$, we see that  $\Delta p=0$ in $\Omega$ and that 
$$
\nabla p \cdot n^\ell = \Delta u \cdot n^\ell -\lambda u \cdot n^\ell
$$
on $\partial\Omega_\ell$, if $\ell $ is sufficiently large. 
Moreover, since $\text{\rm div}(u)=0$ in $\Omega$, we have 
\begin{equation}
\Delta u \cdot n^\ell
=\Delta u^i \cdot n_i^\ell
=(n_i^\ell \partial_j - n_j^\ell \partial_i ) \partial_j u^i,
\end{equation}
where the repeated indices $i, j$  are summed from $1$ to $d$.
As a result, $p$ is a smooth solution of the Neumann problem \eqref{NP1} in $\Omega_\ell$ with
$$
g_{ij}= \partial_ju^i  \quad \text{ and } \quad h= -\lambda u \cdot n^\ell.
$$
It follows from  Theorem  \ref{thm-N1} and Remark \ref{re-N} that  if $B(x, 2 r) \subset \Omega_\ell $, then
\begin{equation}\label{P-100}
\left( \fint_{B(x, r)}
|p-\fint_{B(x, r)}  p |^q \right)^{1/q}
\le C M_q^\ell (|\nabla u|) (\delta_\ell (x))
+ C |\lambda| \| u \|_{L^\infty(\partial\Omega_\ell)},
\end{equation}
where $C$ is independent of $\ell$.
Here, the function $M_q^\ell $ is the analog of $M_q$ for the domain $\Omega^\ell$ and
$\delta_\ell (x) =\text{dist} (x, \partial\Omega_\ell)$.
Under the assumptions that 
$u\in C^1(\overline{\Omega}; \C^d)$ and $u=0$ on $\partial\Omega$,
we  see that   $M_q^\ell (|\nabla u|)(\delta_\ell(x)) \to M_q (|\nabla u|)(\delta(x))$ and
$\|u\|_{L^\infty(\partial\Omega_\ell)} \to 0$,  as $\ell \to\infty$.
Consequently, the estimate \eqref{P-99} follows from \eqref{P-100} by letting $\ell \to \infty$.

The proof of Theorem \ref{thm-P2} is similar.
By Theorem \ref{thm-N1} and Remark \ref{re-N}, 
 for $x_\ell \in \partial\Omega_\ell $ and $0<r< r_0$,
\begin{equation}\label{P3-000}
\aligned
 & \left(\fint_{B(x_\ell , r) \cap \partial\Omega_\ell }
| p  -\fint_{B(x_\ell , r) \cap \partial\Omega_\ell } p  |^q \right)^{1/q}
+
\left(\fint_{B(x_\ell , r) \cap \Omega_\ell}
| p  -\fint_{B(x_\ell , r) \cap\partial \Omega_\ell }  p  |^q \right)^{1/q}\\
& \le C  M_q^\ell (|\nabla u|) (r)  + C |\lambda| \| u \|_{L^\infty(\partial\Omega_\ell)}.
\endaligned
\end{equation}
Let $x_0\in  \partial\Omega$. Choose a sequence $\{ x_\ell \}  $ such that $x_\ell \in \partial\Omega_\ell $ and $x_\ell \to x_0$.
Using the assumption $p\in C(\overline{\Omega})$, a simple limiting argument yields \eqref{P-97}.
\end{proof}

\begin{remark}\label{app-re}
Without the smoothness assumptions that $u\in C^1(\overline{\Omega}, \C^d)$ and $p \in C(\overline{\Omega}; \C)$,
the estimates in Theorems \ref{thm-P1} and \ref{thm-P2} continue to hold if  $q>d$ and $\Omega$ is a bounded $C^1$ domain in $\R^d$ for $d\ge 3$, or
if $2< q< 2+\e$ and $\Omega$ is a bounded Lipschitz domain in $\R^2$.
To see this, note that $u \in W^{1, q}(\Omega; \C^d)$ for $q>d$ implies that $u$ is continuous up to the boundary.
As a result, $\| u\|_{L^\infty(\partial \Omega_\ell) } \to 0$.
Also, by using  \eqref{B1-5}, we may show that  $M_q^\ell (|\nabla u|) (\delta_\ell (x)) \to M_q(|\nabla u|)(\delta(x)$ and
the left-hand side of \eqref{P3-000} converges to the left-hand side of \eqref{P-97}.
\end{remark}

Next, we treat  the case of exterior domains.

\begin{thm}\label{thm-P1e}
Let $\lambda\in \Sigma_\theta$ and $q\ge 2$.
Let $(u, p)$ be a weak solution of \eqref{eq-0}, where
$F\in C_{0, \sigma}^\infty(\Omega)$ and $\Omega$ is an exterior  Lipschitz domain.
Suppose $u\in C^1(\overline{\Omega} ; \C^d)$,  $p \in C(\overline{\Omega}; \C)$, 
and $p(x) =O(1)$ as $|x|\to \infty$.
Let  $B(x, 2r)\subset \Omega$.
Then
\begin{equation}\label{P-99e}
\left(
\fint_{B(x, r)} |p-\fint_{B(x, r)}  p |^q \right)^{1/q}
\le C  \left\{ M_q(|\nabla u|) (\delta (x)) + \|\nabla u \|_{L^q(\partial\Omega)} \right\},
\end{equation}
where $C$ depends only on $d$, $q$ and $\Omega$.
\end{thm}

\begin{thm}\label{thm-P2e}
Let $(u, p)$ be the same as in Theorem \ref{thm-P1e}.
Suppose $x_0 \in \partial\Omega$ and $0< r<  r_0$.
Then
\begin{equation}\label{P-98}
\left(
\fint_{B(x_0, r)\cap \Omega} |p-\beta|^q \right)^{1/q}
+\left(\fint_{B(x_0, r)\cap \partial\Omega}
|p-\beta|^q \right)^{1/q} 
\le C  M_q(|\nabla u|)(r),
\end{equation}
where $\beta=\fint_{B(x_0, r) \cap \partial\Omega} p$ and  $C$ depends only on $d$, $q$ and $\Omega$.
\end{thm}

\begin{proof}

To prove Theorems \ref{thm-P1e} and \ref{thm-P2e}, we first note that 
since $p$ is harmonic in $\Omega$, the condition $p(x)=O(1)$ at infinity implies 
that $ p(x) \to \beta$ for some $\beta \in \C$, as $|x| \to \infty$.
By replacing $p$ with $p-\beta$, we may assume that $p(x) =o(1)$
at infinity. This allows us to apply estimates obtained in Section \ref{section-Ne} for the exterior domain.
As a result, estimates \eqref{P-99e}-\eqref{P-98} follow from \eqref{P1-0e}, 
\eqref{N2-0},  \eqref{N2-1} and Remark \ref{re-EN} by an approximation argument, as in the case of bounded domains.
We should point out that since div$(u)=0$ in $\Omega$ and $u=0$ on $\partial\Omega$, the condition
$$
\int_{\partial\Omega^\ell} u\cdot n^\ell=0
$$
is satisfied for $d\ge 2$.
This follows by applying the divergence theorem in the bounded domain $\Omega \setminus \overline{\Omega_\ell}$.
\end{proof}

%%%%%%%%%%%%%%%%%%%%%%%%

\section{Proof of Theorem \ref{main-1} for $d\ge 3$ }\label{section-M1}

Let $\Omega$ be a bounded $C^{1}$ domain and $\lambda\in \Sigma_\theta$.
We may assume $|\lambda|> C r_0^{-2}$, as the case $|\lambda| \le C r_0^{-2}$ is given by \eqref{V5-3}, which follows from the $L^2$ 
resolvent estimate and local regularity estimates in $C^1$ domains.
We first consider the case $F\in C_{0, \sigma}^\infty(\Omega)$.
Note that by Lemma \ref{lemma-C-2}, 
\begin{equation}\label{main-B1}
M_q (|\nabla u|) (r)\le C \left\{ |\lambda|^{1/2} \| u \|_{L^\infty(\Omega)} + |\lambda|^{-1/2}  \| F \|_{L^\infty(\Omega)} \right\}, 
\end{equation}
for any $2\le q< \infty$ and  $|\lambda|^{-1/2}\le r< r_0$.
 This, together with  \eqref{V5-00}, \eqref{P-99} and \eqref{P-97} with $q=2d$ as well as Remark \ref{app-re}, shows   that for any $N\ge 2$,
\begin{equation}\label{M-1}
\aligned
  |\lambda| \| u \|_{L^\infty (\Omega)}
   \le C_0N^{{4d}+8} \| F \|_{L^\infty (\Omega)}
 + C_0 N^{-1/2} |\lambda | \| u \|_{L^\infty(\Omega)}, 
 \endaligned
 \end{equation}
 where $C_0$ depends only on $d$ and   $\Omega$.
 By choosing $N$ so large that $C_0N^{-1/2} \le (1/2)$, we obtain \eqref{est-0}.

  The estimates for the general case $F\in L^\infty_\sigma(\Omega)$ follow  by an approximation argument, as in \cite{AG-2013, AG-2014}.
Let $(u, p)$ be the energy solution of \eqref{eq-0} with $\int_\Omega p=0$.
Choosing a sequence of functions $\{ F_\ell \}$ in $C^\infty_{0,\sigma}(\Omega)$ such that
$F_\ell \to F$ a.e. in $\Omega$ and $\|F_\ell \|_{L^\infty(\Omega)} 
\le C \| F \|_{L^\infty(\Omega)}$.
Let $(u_\ell, p_\ell)$ be the solution of \eqref{eq-0} with $F_\ell $ in the place of $F$ and with $\int_\Omega p_\ell=0$.
 Then
 \begin{equation}\label{B-1}
 |\lambda|  \| u_\ell \|_{L^\infty(\Omega)}
 \le C \| F \|_{L^\infty(\Omega)},
 \end{equation}
 Note that $\| F_\ell -F \|_{L^2(\Omega)} \to 0$ by the dominated convergence theorem.
 It  follows from the $L^2$ resolvent estimate that 
 $u_\ell \to u$ in $W^{1, 2}(\Omega, \C^d)$ and 
 $p_\ell \to p$ in $L^2(\Omega)$.
 This implies that there exists a subsequence, still denoted by $\{( u_\ell, p_\ell)  \}$, such that
 $u_\ell \to u$, $\nabla u_\ell \to \nabla u$ and $p_\ell \to p$ a.e. in $\Omega$.
 As a result, by taking the limits in \eqref{B-1}, we obtain \eqref{est-0} for any $F\in L^\infty_\sigma(\Omega)$.
This gives the existence in Theorem \ref{main-1}.
Since $L^\infty_\sigma (\Omega)\subset L^2(\Omega; \C^d)$, the uniqueness follows from the uniqueness for the case $q=2$.

If $\Omega$ is $C^{1, \alpha}$ for some $\alpha>0$,
the estimate \eqref{est-1} for $\nabla u$ follows from \eqref{est-0} and the local Lipschitz estimate for the Stokes equations.
See Remark \ref{re-smooth}.

%%%%%%%%%%%%%%%%%%

\section{Proof of Theorem \ref{main-2} for $d\ge 3$}\label{section-M2}

Throughout this section we assume that $\Omega$ is an exterior $C^1$  domain in $\R^d$.

\begin{lemma}\label{lemma-ed1}
Let $(u, p)$ be a weak solution of \eqref{eq-0} with
 $F\in C_{0,\sigma}^\infty(\Omega)$.
 Then $|\nabla p(x)| = O(|x|^{-d+\e})$ as $|x|\to \infty$, for any $\e>0$.
 Consequently, $p(x) \to \beta$ as $|x|\to \infty$ for some $\beta \in \C$.
\end{lemma}

\begin{proof}
Since $F\in L^q(\Omega; \C^d)$ for any $q>1$, 
it follows from the $L^q$ resolvent estimates in $C^1$ domains \cite{GS-2023} that
$u\in W^{1, q}_0(\Omega; \C^d)$ and $p\in L^q_{loc}(\overline{\Omega}; \C)$ for any $q>1$.
Fix $R\ge 1$ so that $\Omega^c \subset B(0, R) $.
Choose $\varphi\in C^\infty(\R^d)$ such that $\varphi=1$ in $B(0, 3R)^c$ and $\varphi =0$ in $B(0, 2R)$.
Note that 
$$
\left\{
\aligned
-\Delta (u \varphi) + \nabla (p \varphi) + \lambda (u \varphi) & =h,\\
\text{\rm div}(u\varphi) & =g
\endaligned
\right.
$$
in $\R^d$, where $g, h$ are given by \eqref{L-14}.
By the generalized  $L^q$ resolvent estimates in $\R^d$ \cite{Sohr-1994}, we see that $\nabla (p \varphi)  \in L^q(\R^d; \C^d )$ for any $q>1$.
It follows that  $\nabla p \in L^q(B(0, 3R)^c; \C^d)$.
Since $p$ is harmonic in $\Omega$, by the mean value property, we obtain
$$
|\nabla p(x) |
\le \left( \fint_{B(x, r)} |\nabla p|^q \right)^{1/q}
\le C r^{-d/q} \| \nabla p \|_{L^q(B(0, 3R)^c)},
$$
as long as  $B(x, r) \subset B(0, 3R)^c$.
This implies that $|\nabla p(x)| = O(|x|^{-d + \e})$ as $|x|\to \infty$ for any $\e>0$.
Consequently, by applying the fundamental theorem of calculus, it follows  that 
$p$ is bounded at $\infty$.
In view of  Lemma \ref{lemma-Ne0}, we further deduce that $p(x) \to \beta$
as $|x| \to \infty$ for some $\beta \in \C$.
\end{proof}

\begin{lemma}\label{lemma-ed2}
Let $u\in W^{1, 2}_{\loc}(\overline{\Omega}; \C^d)$ and $p\in L^2_{\loc}(\overline{\Omega}; \C)$.
Suppose $(u, p)$ satisfies \eqref{eq-0} in the sense of distributions for some $F\in L^\infty_\sigma (\Omega)$.
Also assume that $|\nabla p(x) |= O(|x|^{-1})$ as $|x|\to \infty$.
Then $p(x) \to \beta $  for some $\beta \in \C$ as $|x|\to \infty$.
\end{lemma}

\begin{proof}
Since div$(F)=\text{\rm div}(u) =0$ in $\Omega$, it follows that $p$ is harmonic in $\Omega$.
Also note that if $\Omega^c \subset B(0, R-1)$, then
$$
\int_{\partial B(0, R)} \nabla p \cdot n 
=\int_{\partial B(0, R)}
\left( \Delta u \cdot n -\lambda u \cdot n + F\cdot n\right) =0, 
$$
where we have used the observation $\Delta u \cdot n =(n_i\partial_j -n_j \partial_i) \partial_j u^i$.
Let
$$
\phi (R) = \fint_{\partial B(0, R)}  p =\fint_{\partial B(0, 1)} p (R\omega)\, d\omega
$$
for $R$ large. Since
$$
\phi^\prime (R) =\fint_{\partial B(0, R)} \nabla p \cdot n =0, 
$$
we deduce that $\phi (R)$ is constant.
This implies that
$$
\fint_{B(0, 2R)\setminus B(0, R)} p 
$$
is constant for $R$ large.  Using 
$$
\left(\fint_{B(0, 2R)\setminus B(0, R)} |  p|^2 \right)^{1/2}
\le C R \left(\fint_{B(0, 2R)\setminus B(0, R)} |\nabla p|^2 \right)^{1/2}
+ \big| \fint_{B(0, 2R)\setminus B(0, R)} p \big|,
$$
we see that $p(x) = O(1)$ as $\beta \to \infty$.
It follows that $p(x)\to \beta$ for some $\beta\in \C$ as $|x|\to \infty$.
\end{proof}

\begin{proof}[Proof of Theorem \ref{main-2} for $d\ge 3$]

Let $(u, p)$ be the weak solution of \eqref{eq-0} with $F\in C_{0, \sigma}^\infty(\Omega)$.
In view of Lemma \ref{lemma-ed1}, by replacing $p$ with $p-\beta$, we may assume $p(x) \to 0$
as $|x| \to \infty$.
As a result, the estimates \eqref{P-99e} and \eqref{P-98} hold.
Also, note that \eqref{main-B1} continues to hold in an exterior $C^1$ domain.
Consequently,   by Remark \ref{re-V}, if $|\lambda|> 16 r_0^{-2} N^4$, then
\begin{equation}\label{M-2}
   | \lambda| \| u \|_{L^\infty (\Omega)}
   \le C_0N^{{4d}+8} \| F \|_{L^\infty (\Omega)}
 + C_0 N^{-\frac12}
  |\lambda| \| u \|_{L^\infty(\Omega)}, 
 \end{equation}
 where $C_0$ depends only on $d$ and $\Omega$.
 We now fix $N=N_0$ so that $C_0 N_0^{-1/2} \le (1/2)$.
 Let $\lambda_0= 16 r_0^{-2} N_0^4$.
 It follows that if $\lambda \in \Sigma_\theta$ and $|\lambda|> \lambda_0$,  then
 \begin{equation}\label{M-3}
  | \lambda| \| u \|_{L^\infty (\Omega)}
 \le 2C_0 N_0^{4d+8} \| F \|_{L^\infty(\Omega)}.
 \end{equation}

The existence of solutions for the general $F\in L^\infty_\sigma(\Omega)$ follows by an approximation argument as in \cite{AG-2014, AGH-2015}.
Choosing a sequence of functions $ \{ F_\ell \}$  in $ C_{0, \sigma}^\infty (\Omega)$ such that
$F_\ell \to F$ a.e. in $\Omega$ and $\| F_\ell \|_{L^\infty(\Omega)} \le C \| F \|_{L^\infty(\Omega)}$.
Let $(u_\ell, p_\ell)$ be the weak solution of \eqref{eq-0} with $F_\ell$ in the place of $F$ and with the property that 
$p_\ell (x) = o(1)$ at infinity.
Then 
\begin{equation}\label{M-4}
|\lambda| \| u_\ell \|_{L^\infty (\Omega)} \le C \| F_\ell \|_{L^\infty(\Omega)} \le C \| F \|_{L^\infty(\Omega)},
\end{equation}
where $C$ is independent of $\ell$.
Choose a subsequence, still denoted by $(u_\ell, p_\ell)$, such that
$u_\ell \to u $ weakly in $W^{1, 2}_{\loc}(\Omega; \C^d)$,
$u_\ell \to u$ a.e. in $\Omega$,
$p_\ell \to p$ weakly in $L^2_{\loc}(\Omega; \C)$,
and $p_\ell  \to p$ a.e. in $\Omega$.
Since $p_\ell$ is harmonic, we also have $\nabla p_\ell  \to \nabla p$ in $\Omega$.
It follows that $(u, p)$ satisfies \eqref{eq-0} in the sense of distributions.
By taking the limit in \eqref{M-4}, we obtain $|\lambda| \| u \|_{L^\infty(\Omega)}
\le C \| F \|_{L^\infty(\Omega)}$.
To show $p(x)= \beta + o(1)$ at infinity for some $\beta \in \C$, we note that by Theorem \ref{thm-P1e} and \eqref{main-B1},
\begin{equation}\label{M-5}
\delta (x) |\nabla p_\ell (x) |
\le C  \| \nabla u_\ell \|_{L^q(\partial\Omega)}
\le C |\lambda|^{-1/2} \| F \|_{L^\infty(\Omega)},
\end{equation}
where $q>d$ and $\delta (x) \ge 1$.
By letting $\ell \to \infty$, we obtain 
$$
\delta(x) |\nabla p(x)| \le C |\lambda|^{-1/2} \| F \|_{L^\infty(\Omega)}.
$$
By Lemma \ref{lemma-ed2}, this implies $p(x) \to \beta$ as $|x|\to \infty$ for some $\beta\in \C$.

To show the uniqueness.
Suppose  $u\in W_0^{1, \infty}(\Omega; \C^d)$ and  $p\in L^2_{\loc} (\overline{\Omega}; \C) $ satisfies the condition 
$p(x)= O(1)$ at infinity.
Assume that $(u, p)$ is a solution of $\eqref{eq-0}$ in $\Omega$ with $F=0$.
Since  $p\to \beta$ for some $\beta \in \C$ as $|x|\to \infty$,
by replacing $p$ with $p-\beta$ we may assume that $p (x) \to 0$ as $|x|\to \infty$.
As a result, the  argument for the existence of solutions with  $F\in C_{0, \sigma}^\infty(\Omega)$
yields the estimate  \eqref{M-3} with $F=0$.
Consequently, $u=0$ in $\Omega$.
\end{proof}

%%%%%%%%%%%%%%%%%%%%%%%%

\section{A perturbation argument for Lipschitz domains}\label{section-T}

Let $\H_\psi$ be the region above a Lipschitz graph,  given by \eqref{H}, 
where $\psi: \R^{d-1} \to \R$ is a Lipschitz function such that $\psi (0)=0$ and $\|\nabla^\prime \psi  \|_\infty\le M$.
Let $A_\psi^q, B_\psi^q$ be two Banach spaces on $\H_\psi$, defined by \eqref{AB} for $1< q< \infty$.

\begin{lemma}\label{lemma-T2}
Let $ \lambda\in \Sigma_\theta$ with $|\lambda|=1$.
For any $F\in L^2(\H_\psi; \C^d)$, $f\in L^2(\H_\psi; \C^{d\times d} )$ and $g \in B_\psi^2$,
there exists  a unique  $(u, p)$ such that $u \in W_0^{1, 2}(\H_\psi; \C^d)$, 
$p\in A_\psi^2 $, and \eqref{G-eq} holds.
Moreover, the solution satisfies 
\begin{equation}\label{T2-0}
\aligned
 &  \| \nabla u \|_{L^2(\H_\psi)}
+\| u \|_{L^2(\H_\psi)} + \| p \|_{A_{\psi}^2}
\\
&\quad \le C \left\{
\| F \|_{L^2(\H_\psi)}
+   \| f \|_{L^2(\H_\psi)}
+  \| g \|_{L^2(\H_\psi)}
+   \| g \|_{\hW^{-1, 2}(\H_\psi)}
\right\},
\endaligned
\end{equation}
where $C$ depends only on $\theta$ and $M$.
\end{lemma}

\begin{proof}

Step 1. Establish the existence and uniqueness in the special case $g=0$.

Consider the bilinear form,
$$
B[u, v]=\int_{\H_\psi} \nabla u \cdot \overline{\nabla v} + \lambda \int_{\H_\psi} u \cdot \overline{v},
$$
in the Hilbert space,
$$
V=\left\{ u \in W^{1, 2}_0(\H_\psi; \C^d): \ \text{\rm div}(u) =0 \text{ in } \H_\psi \right\}.
$$
By applying the Lax-Milgram theorem, we deduce  that there exists a unique $u\in V$ such that
$$
B[u, v] =\int_{\H_\psi}  F \cdot  \overline{v}  -\int_{\H_\psi} f \cdot \overline{\nabla v}
$$
for any $v\in V$, and that $u$ satisfies 
\begin{equation}\label{T2-2}
 \|\nabla u \|_{L^2(\H_\psi)}
 + \| u \|_{L^2(\H_\psi)}
 \le C \left\{ \| F \|_{L^2(\H_\psi)} + \| f \|_{L^2(\H_\psi)}\right\},
 \end{equation}
 where $C$ depends only on $d$ and $\theta$.
It follows that there exists $p\in L^2_{\loc} (\H_\psi)$ such that
\begin{equation}\label{T2-3}
-\Delta u + \lambda u  +\nabla p = F +\text{\rm div}(f)
\end{equation}
holds in $\H_\psi $ in the distributional sense.
To show $p \in A_\psi^2$,  observe that by \eqref{T2-3} and \eqref{T2-2},  $\nabla p \in W^{-1, 2}(\H_\psi; \C^d)$ and
\begin{equation}\label{T2-4}
\|\nabla p \|_{W^{-1, 2}(\H_\psi)}
\le C \left\{ \| F \|_{L^2(\H_\psi)} + \| f \|_{L^2(\H_\psi)} \right\}.
\end{equation}
For  $\widetilde{g} \in B_\psi^2$, let $ \widetilde{w}$ be a function  in $ W_0^{1, 2}(\H_\psi; \C^d)$ such that
$\text{\rm div}(\widetilde{w}) =\widetilde{g}$ in $\H_\psi$ and
$$
\|\widetilde{w} \|_{W^{1, 2}_0(\H_\psi)}
\le C  \| \widetilde{g}\|_{B_\psi^2},
$$
 given by Theorem  \ref{thm-A2}.
Note that
$$
\aligned
|\langle p, \widetilde{g} \rangle |
& = |\langle p, \text{\rm div}(\widetilde{w}) \rangle |
= |\langle \nabla p, \widetilde{w}  \rangle |\\
& \le \| \nabla p \|_{W^{-1, 2} (\H_\psi)}
\| \widetilde{w}  \|_{W^{1, 2}_0(\H_\psi)}\\
&\le C \left\{ \| F \|_{L^2(\H_\psi)} + \| f \|_{L^2(\H_\psi)} \right\}
\left\{ \| \widetilde{g} \|_{L^2(\H_\psi)} + \| \widetilde{g} \|_{\hW^{-1, 2}(\H_\psi)} \right\},
\endaligned
$$
where we have used \eqref{T1-0} and \eqref{T2-4}.
Since $(A_\psi^2)^\prime= B_\psi^2$, by duality, we obtain 
$$
\| p \|_{A_\psi^2} \le C\left\{  \| F \|_{L^2(\H_\psi)} + \| f \|_{L^2(\H_\psi)} \right\}.
$$
This gives the existence of solutions with the estimate \eqref{T2-0} for the case $g=0$.
We point out that the uniqueness in the general case  follows from the Lax-Milgram theorem.

\medskip

Step 2. Establish the existence and the estimate \eqref{T2-0} in the general case $g\in B_\psi^2$.

Given $g \in B_\psi^2$, let $w$ be the function given by Theorem \ref{thm-A2}.
Let $(v, p)$ be the solution of 
$$
\left\{
\aligned
-\Delta v + \lambda v +\nabla p & = F -\lambda w +\Delta w +\text{\rm div}(f),\\
\text{\rm div} (v) & =0, 
\endaligned
\right.
$$
in $\H_\psi$, obtained in Step 1. Let $u=v+w$. Then $(u, p)$ is a solution of \eqref{eq-0} and
$$
\aligned
\|\nabla u \|_{L^2(\H_\psi)}
+ \| u \|_{L^2(\H_\psi)}
& \le \| \nabla v\|_{L^2(\H_\psi)} + 
 \| \nabla w\|_{L^2(\H_\psi)} + 
 \|   v\|_{L^2(\H_\psi)} + 
 \| w\|_{L^2(\H_\psi)} \\
 & \le C \left\{
 \| F -\lambda w \|_{L^2(\H_\psi)}
 + \|\nabla w \|_{L^2(\H_\psi)}
 + \| f \|_{L^2(\H_\psi)}
 + \| w \|_{L^2(\H_\psi)} \right\}\\
 & \le 
 C  \left\{
\| F \|_{L^2(\H_\psi)}
+   \| f \|_{L^2(\H_\psi)}
+  \| g \|_{L^2(\H_\psi)}
+   \| g \|_{\hW^{-1, 2}(\H_\psi)}
\right\}, 
 \endaligned
 $$
where we have used \eqref{T1-0} for the last step.
\end{proof}

\begin{lemma}\label{lemma-T3}
Let $(X_0, X_1)$, $(Y_0, Y_1)$ be interpolation pairs of Banach spaces, and  $S: X_j \to Y_j$, 
$j=0, 1$ a bounded linear operator.
Let $X_t =[X_0, X_1]_t$ and $Y_t =[Y_0, Y_1]_t$ for $0< t< 1$.
If $S^{-1}: Y_t \to X_t$ exists and is bounded for some $0< t<1$, there exists $\e>0$ such that
$
S^{-1} : Y_s\to X_s
$
exists and is bounded for $s\in (t-\e, t+\e)$.
\end{lemma}

\begin{proof}
See \cite[Proposition 4.1]{VV1988}.
\end{proof}

\begin{thm}\label{thm-T}
Let $ \lambda\in \Sigma_\theta$ and $M>0$.
There exists $\e \in (0, 1)$,  depending only on $d$, $\theta$ and $M$, such that 
if  $\|\nabla^\prime \psi \|_\infty\le  M$ and $|\frac{1}{q}-\frac12| < \e$, then 
for any $F\in L^q(\H_\psi; \C^d)$, $f\in L^q(\H_\psi; \C^{d\times d} )$ and $g \in B_\psi^q$,
there exists  a unique  $(u, p)$ such that $u \in W_0^{1, q}(\H_\psi; \C^d)$, 
$p\in A_\psi^q $, and \eqref{G-eq} holds.
Moreover, the solution satisfies 
\begin{equation}\label{T-est}
\aligned
 & |\lambda|^{1/2} \| \nabla u \|_{L^q(\H_\psi)}
+|\lambda| \| u \|_{L^q(\H_\psi)}
\\
&\quad \le C \left\{
\| F \|_{L^q(\H_\psi)}
+ |\lambda|^{1/2}  \| f \|_{L^q(\H_\psi)}
+|\lambda|^{1/2}  \| g \|_{L^q(\H_\psi)}
+  |\lambda| \| g \|_{\hW^{-1, q}(\H_\psi)}
\right\},
\endaligned
\end{equation}
where $C$ depends only on $d$, $\theta$ and $M$.
\end{thm}

\begin{proof}
By dilation we may assume $|\lambda|=1$.
Let $\psi: \R^{d-1} \to \R$ be a Lipschitz function with $\|\nabla^\prime \psi \|_\infty\le M$.
As in \cite{GS-2023}, we introduce two Banach spaces,
\begin{equation}\label{XY}
X_\psi^q =W^{1, q}_0 (\H_\psi; \C^d) \times A_{\psi}^q
\quad \text{ and } \quad
Y_\psi^q= W^{-1, q}(\H_\psi; \C^d) \times B_\psi^q,
\end{equation}
where $1<q<\infty$.
Consider the linear operator,
\begin{equation}\label{S}
S_\psi^\lambda (u, p)
=\left(-\Delta u +\nabla p + \lambda u, \text{\rm div}(u) \right).
\end{equation}
Then $S_\psi^\lambda$ is bounded from $X_\psi^q$ to $Y_\psi^q$
for any $1<q<\infty$. Moreover,
\begin{equation}\label{S-1}
\| S_\psi^\lambda \|_{X_\psi^q \to Y_\psi^q} \le C,
\end{equation}
where $C$ depends only on $d$, $q$ and $M$.
Note that the existence and uniqueness of solutions of \eqref{G-eq} in $X_\psi^q$,   with data $F\in L^q(\H_\psi; \C^d)$,
$f\in L^q(\H_\psi; \C^{d\times d} )$ and $g \in B_\psi^q$, are
equivalent to the invertibility of the operator $S^\lambda_\psi: X_\psi^q \to Y_\psi^q$.
In \cite{GS-2023} we were able to show  that  there exists $c_0=c_0(d, q, \theta )>0$ such that 
if $\|\nabla^\prime \psi \|_\infty< c_0$, then $S_\psi^\lambda: X_\psi^q \to Y_\psi^q$ is invertible.
Here, we will show that there exists $\e =\e(d, \theta, M)>0$ such that if $\|\nabla^\prime \psi \|_\infty\le M$ and
$|\frac{1}{q}-\frac12 |< \e$, then
$S_\psi^\lambda: X_\psi^q \to Y_\psi^q$ is invertible.
Moreover,  
\begin{equation}\label{S-1a}
\| \left( S_\psi^\lambda \right)^{-1} \|_{Y^q_\psi\to X_\psi^q} \le C,
\end{equation}
where $C$ depends only on $d$, $\theta$ and $M$.
This gives Theorem \ref{thm-T}.

To this end, we choose $q_0=(4/3)$ and $q_1=4$ so that  $1<q_0< 2< q_1< \infty$ and
$|\frac{1}{q_0} -\frac12 | = |\frac{1}{q_1} -\frac12 | = \frac14$.
Note that $(X_\psi^{q_0}, X^{q_1}_\psi)$ and $(Y_\psi^{q_0}, Y_\psi^{q_1})$ are (complex) interpolation pairs 
of Banach spaces. In fact,  for $0< t< 1$,
$$
[X_\psi^{q_0}, X_\psi^{q_1}]_t =  X_\psi^q \quad \text{ and } \quad
[Y_\psi^{q_0}, Y_\psi^{q_1}]_t =  Y_\psi^q, 
$$
where $\frac{1}{q} =\frac{1-t}{q_0} +\frac{t}{q_1}$. 
As we pointed out earlier, the linear operator $S_\psi^\lambda$ is bounded from $X_\psi^q$
to $Y_\psi^q$ for $1< q< \infty$.
By Lemma \ref{lemma-T2},  $S^{-1}: Y_\psi^2 \to X^2_\psi$ exists and is bounded.
It follows by Lemma \ref{lemma-T3} that  there exists $\e>0$ such that 
$S^{-1}: Y_\psi^q\to X_\psi^q$ exists and is bounded if $|\frac{1}{q}-\frac12| < \e$.
Moreover,  observe that $\| S_\psi^\lambda \|_{X_\psi^{q_j}\to Y_\psi^{q_j}}\le C$ for $j =0, 1$ and
$\| (S_\psi^\lambda)^{-1} \|_{Y_\psi^2 \to X_\psi^2} \le C$, where $C$ depends only on $d$, $\theta$ and $M$.
As a result, it is possible to choose $\e$ and $C_0$, depending only on $d$, $\theta$ and $M$, such that \eqref{S-1a} 
holds for all $q$ satisfying $|\frac{1}{q}-\frac12|< \e$.
This follows from the same argument as that for the $L^p$ spaces in \cite[p.390]{VV1988}.
\end{proof}

The following theorem covers the two dimensional case for Theorem \ref{main-1}.

\begin{thm}\label{thm-T4}
Let $\lambda\in \Sigma_\theta$, where $\theta\in (0, \pi/2)$.
Let $\Omega$ be a bounded Lipschitz domain in $\R^2$.
For any $F \in L^\infty_\sigma (\Omega)$, the unique weak solution of \eqref{eq-0} in
$W^{1, 2}(\Omega; \C^2) \times L^2(\Omega; \C)$ with $\int_\Omega p=0$ satisfies the estimate,
\begin{equation}\label{T4-0}
|\lambda| \| u \|_{L^\infty (\Omega)}
\le C \| F \|_{L^\infty(\Omega)},
\end{equation}
where $C$ depends only on $\theta$ and the Lipschitz character of $\Omega$
Moreover, if $\Omega$ is a bounded $C^{1, \alpha}$ domain in $\R^2$ for some $\alpha>0$, then
\begin{equation}\label{T4-1}
|\lambda|^{1/2}  \|\nabla u \|_{L^\infty(\Omega)}
\le C \| F \|_{L^\infty(\Omega)},
\end{equation}
where $C$ depends on $\theta$ and $\Omega$.
\end{thm}

\begin{proof}
As in the case $d\ge 3$,
the gradient estimate \eqref{T4-1} for $C^{1, \alpha}$ domains follows from \eqref{T4-0} and \eqref{Lip-est}.
With Theorem \ref{thm-T} at our disposal, the proof for \eqref{T4-0}
 is similar to that for the case of $C^1$ domains. We give its outlines.

\medskip

Step 1. Since $\Omega$ is Lipschitz, there exist $M>0$ and $r_0>0$ with the properties that  for any $x_0\in \partial\Omega$,
there is a Cartesian coordinate system with origin $x_0$ such that
$$
B(x_0, 16r_0)\cap \Omega
=B(x_0, 16 r_0) \cap \H_{\psi} \quad \text{ and } \quad
B(x_0, 16 r_0)\cap \partial\Omega
= B(x_0, 16 r_0) \cap \partial \H_\psi,
$$
where $\psi: \R \to \R$ is a Lipschitz function with $ \| \nabla^\prime \psi \|_\infty \le M$.

\medskip

Step 2.  Choose $q\in (2, 2+\e)$, where $\e>0$ is given by Theorem \ref{thm-T} and depends only on $M$.
We also assume that $q$ is close to $2$ so that the boundary estimate \eqref{C-1-2}
holds. See Remark \ref{Lip-B-re}.

\medskip

Step 3. Observe that since  $q>d=2$, 
the same argument as in Sections \ref{section-L} and \ref{section-V} for $C^1$ domains, shows that 
the estimate \eqref{V5-00} continues to hold for the two-dimensional Lipschitz domain $\Omega$.

\medskip

Step 4. Note that the estimates for the pressure in Section \ref{section-P1} are proved for Lipschitz domains in $\R^d$, $d\ge 2$.
By combining \eqref{V5-00} with \eqref{P3-0}-\eqref{P3-00}, we obtain
\begin{equation}
|\lambda| \| u \|_{L^\infty(\Omega)}
\le C N^{16} \| F \|_{L^\infty(\Omega)}
+ C N^{\frac{2}{q}-1} |\lambda| \| u \|_{L^\infty(\Omega)}
\end{equation}
for any $N\ge 2$.
This gives the estimate
\eqref{T4-0}
 under the additional assumption that $F \in C_{0, \sigma}^\infty (\Omega)$.
An approximation argument, similar to that in the case of $C^1$ domains,  yields \eqref{T4-0}
for $F\in L^\infty_\sigma (\Omega)$.
\end{proof}

\begin{thm}\label{thm-T5}
Let $\Omega$ be a bounded Lipschitz domain in $\R^2$.
Let $1< q< \infty$ and $\lambda\in \Sigma_\theta$, where $\theta \in (0, \pi/2)$.
Then for any $F \in L^q_\sigma (\Omega; \C^2)$, the Dirichlet problem \eqref{eq-0} has a unique 
solution in $W^{1, q}_0 (\Omega; \C^2) \times L^q(\Omega; \C)$ with $\int_\Omega p =0$.
Moreover, the solution satisfies the estimate
\begin{equation}\label{T5-0}
|\lambda | \| u \|_{L^q(\Omega)}
\le C \| F \|_{L^q(\Omega)},
\end{equation}
where $C$ depends only on $q$, $\theta$ and $\Omega$.
\end{thm}

\begin{proof}
The case for $2< q< \infty$ follows from the case $q=2$ and Theorem \ref{thm-T4}
by interpolation.
The case $1<q< 2$ is given by duality.
\end{proof}

\begin{remark}
As in the case of $C^1$ exterior domains,
one also obtains the  $L^\infty$ resolvent estimate in an exterior Lipschitz domain in $\R^2$  if $\lambda\in \Sigma_\theta$ and
$|\lambda|$ is large. We omit the details.
\end{remark}

 \appendixpage

%%%%%%%%%%%%%%%%%%%%%%

 \appendix

%%%%%%%%%%%%%%%%%%

\section{Regularity for the Stokes equations in Lipschitz and $C^1$ domains}

In this appendix we collect some regularity results for the Stokes equations in nonsmooth domains.
Let $\Omega$ be a bounded Lipschitz domain.
For a function $u$  in $\Omega$, the nontangential maximal function of $u$ is defined by
\begin{equation}\label{n-max}
(u)^* (x) =\sup \left\{ |u(y)|: \  y \in \gamma(x) \right\}
\end{equation}
for $x\in \partial\Omega$, where
$\gamma(x)= \{  y \in \Omega: \  |y-x|< C_0\,  \text{\rm dist}(y, \partial\Omega) \}$ and  $C_0>1$ depends on $d$ and $\Omega$.
We say $u=f$ n.t. on $\partial\Omega$,  if
$u(y) \to f(x)$ as $y \in \gamma(x)$ and $y \to x$  for a.e. $x\in \partial\Omega$.

\begin{thm}\label{thm-B1}
Let $\Omega$ be a bounded Lipschitz domain in $\R^d$, $d\ge 2$ with connected boundary.
There exists $\e\in (0, 1)$, depending only on $d$ and $\Omega$,  with the property that  for any
$f \in W^{1, q}(\partial\Omega; \R^d)$ and $q\in (2-\e, 2+\e)$, there exist a unique $v \in W^{1, q}(\Omega; \R^d)$ 
 and $\pi  \in L^q(\Omega)$ with $\int_\Omega \pi =0$ such that 
\begin{equation}\label{B1-0}
-\Delta v +\nabla \pi =0 \quad \text{ and } \quad \text{\rm div}(v) =0 \quad \text{ in }\Omega,
\end{equation}
 $v=f$ n.t. on $\partial\Omega$, and $(\nabla v)^* \in L^q(\partial\Omega)$.
 Moreover, $\nabla v$ and $\pi$  have  n.t. limits  a.e. on $\partial\Omega$,  and 
  the solution $(v, \pi)$ satisfies 
 \begin{equation}\label{B1-1}
 \| (\nabla v)^* \|_{L^q(\partial\Omega)}
 + \| (\pi )^* \|_{L^q(\partial\Omega)} \le C \| \nabla_{\tan} f \|_{L^q(\partial\Omega)},
\end{equation}
where  $\nabla_{\tan} f$ denotes the tangential gradient of $f$ on $\partial\Omega$ and $C$ depends on $d$, $q$ and $\Omega$.
Furthermore, if $\Omega$ is a bounded $C^1$ domain, the estimate \eqref{B1-1} holds for any $1< q< \infty$.
\end{thm}

\begin{proof}
The proof for the case of Lipschitz domains was given in \cite{FKV-1988}.
See \cite{Mitrea-2004} for the $C^1$ case.
Both papers used the method of layer potentials.
\end{proof}

Let $\{ \Omega_\ell \}$ be a sequence of smooth domains with uniform Lipschitz characters
such that $\Omega_\ell \nearrow \Omega$.
Let $\Lambda_\ell: \partial\Omega \to \partial\Omega_\ell $ be a homeomorphism with the property that 
$\Lambda_\ell (x) \in \gamma (x)$ for $x\in \partial\Omega$ \cite{Verchota-1984}.
It follows from \eqref{B1-1} by the dominated convergence theorem that the solution $(v, \pi)$
in Theorem \ref{thm-B1} satisfies 
\begin{equation}\label{B1-2}
\int_{\partial\Omega}\big \{   |\nabla v \circ \Lambda_\ell  -\nabla v |^q + |\pi \circ \Lambda_\ell -\pi |^q \big\}
\to 0, 
\end{equation}
as $\ell \to \infty$.
We now consider the Dirichlet problem, 
\begin{equation}\label{D-A}
\left\{
\aligned
-\Delta u +\nabla p & = F & \quad & \text{ in } \Omega,\\
\text{\rm div}(u) & =0 & \quad & \text{ in } \Omega,\\
u & =0 & \quad & \text{ on } \partial\Omega,
\endaligned
\right.
\end{equation}
where $F\in C_{0,\sigma}^\infty(\Omega)$.
Let $(w, \phi)\in W^{2, q}(\R^d; \C^d)\times W^{1, q} (\R^d, \C)$ be a solution of
\begin{equation}
-\Delta w +\nabla \phi =F \quad \text{ and } \quad \text{\rm div}(w) =0 \quad \text{ in } \R^d.
\end{equation}
Using the fundamental theorem of calculus, it is not hard to show that
\begin{equation}\label{B1-4}
\int_{\partial\Omega}\big\{   |\nabla w \circ \Lambda_\ell  -\nabla w  |^q + |\phi\circ \Lambda_\ell -\phi|^q \big\}
\to 0, 
\end{equation}
as $\ell \to \infty$.
Note that $(v, \pi)= (u-w, p-\phi)$ is a solution of \eqref{B1-0} in $\Omega$ with boundary data $-w\in W^{1, q}(\partial\Omega; \C^d)$.
It follows from \eqref{B1-2} and \eqref{B1-4} that  $\nabla u$ and $p$ are well defined on $\partial\Omega$ and that 
\begin{equation}\label{B1-5}
\int_{\partial\Omega}\big\{   |\nabla u \circ \Lambda_\ell  -\nabla u |^q + |p\circ \Lambda_\ell -p|^q \big\} 
\to 0,
\end{equation}
as $ \ell \to \infty$,
where $|q-2|< \e$ if $\Omega$ is Lipschitz, and $1<q< \infty$ if $\Omega$ is $C^1$.
The observation \eqref{B1-5} was used in several approximation arguments.

%%%%%%%%%%%%%%

\section{A divergence problem for a Lipschitz graph domain}

In this appendix we provide  a proof for the following theorem, which is used in the proof  of Lemma \ref{lemma-T2}.

\begin{thm}\label{thm-A2}
Let $g \in B_\psi^q  = L^q (H_\psi; \C)\cap \hW^{-1, q}(\H_\psi; \C)$, where $1< q< \infty$ and 
$\psi:\R^{d-1} \to \R$ is a Lipschitz function with $\|\nabla^\prime \psi \|_\infty \le M$.
Then there exists $w \in W^{1,q}_0(\H_\psi; \C^d)$ such that $\text{\rm div} (w) =g$ in $\H_\psi$  in the sense of distributions and 
\begin{equation}\label{T1-0}
\|\nabla w \|_{L^q(\H_\psi)} + 
\| w \|_{L^q(\H_\psi)}
\le C\left\{ \| g \|_{L^q(\H_\psi)} +  \| g \|_{\hW^{-1, q}(\H_\psi)} \right\},
\end{equation}
where $C$ depends only on $d$, $q$ and $M$.
\end{thm}

\begin{lemma}\label{lemma-A2-1}
Let $\Omega$ be a bounded domain in $\R^d$ and $1< q< \infty$.
Suppose that $\Omega$ is  star-like with respect to every point of $B(x_0, r_0)$ for some $B(x_0, r_0)\subset \Omega$.
Then for any $f\in L^q(\Omega)$ with $\int_\Omega f =0$, there exists $u \in W^{1, q}_0(\Omega)$ such that
$\text{\rm div}(u) =f$ in $\Omega$ and 
\begin{equation}\label{A2-00}
\| \nabla u \|_{L^q(\Omega)} \le C \| f\|_{L^q(\Omega)}, 
\end{equation}
\begin{equation}\label{A2-0}
\|  u \|_{L^q(\Omega)}
\le C \| f \|_{\hW^{-1, q}(\Omega)}.
\end{equation}
Moreover, the constants $C$ in \eqref{A2-00}-\eqref{A2-0} satisfy 
\begin{equation}\label{A2-1}
C\le C_0 \left[ \text{\rm diam}(\Omega)/r_0\right]^d 
\left( 1 + \text{\rm diam}(\Omega)/r_0 \right),
\end{equation}
where $C_0$ depends on $d$ and $q$.
\end{lemma}

\begin{proof}
This is a classical result of M.E. Bogovski\u{\i} \cite{Bogo}. 
We point out that although the estimate \eqref{A2-0} is not usually stated for bounded domains,
it follows from the Bogovski\u{\i} formula.
For the reader's convenience, we sketch a proof, following the presentation in \cite[pp.162-167]{Galdi}.

By translation and dilation, we may assume $x_0=0$ and $r_0=1$.
Choose $\omega\in C_0^\infty(B(0, 1))$ such that $\int_{B(0, 1)} \omega =1$.
Let
$$
u(x)=\int_\Omega K(x, y) f(y)\, dy,
$$
where $K(x, y)=(K_1(x, y), \dots, K_d (x, y))$ and
\begin{equation}\label{A2-3}
\aligned
K_i(x, y)
 & =(x_i -y_i) \int_1^\infty \omega (y + t (x-y)) t^{d-1} \, dt\\
 &=
 \frac{x_i -y_i}{|x-y|^d}
\int_0^\infty \omega \Big( x+ t \frac{x-y}{|x-y|}\Big)
\left( |x-y|+ t\right)^{d-1}\, dt.
\endaligned
\end{equation}
Note that
$$
|K(x, y)|\le C |x-y|^{1-d} \quad
\text{ and } \quad
|\nabla_x K(x, y)| + |\nabla_y K(x, y)| \le C |x-y|^{-d}.
$$
Moreover, 
$$
\partial_j  u_i (x)
= f(x)\int_\Omega \frac{(x_j -y_j)(x_i-y_i)}{|x-y|^2} \omega(y)\, dy + T_{ij} (f) (x),
$$
where
$$
T_{ij} (f) (x)=\text{p.v.}  \int_\Omega \frac{\partial}{\partial x_j}
K_i(x, y) f(y)\, dy.
$$
Furthermore,  using \eqref{A2-3}, 
a direct computation shows that
\begin{equation}\label{A2-4}
\aligned
\frac{\partial}{\partial x_j} K_i (x, y)
& =\frac{\delta_{ij}}{|x-y|^d}
\int_0^\infty \omega\Big(x + t \frac{x-y}{|x-y|}\Big) \big(|x-y| + t\big)^{d-1}\, dt\\
&\quad
+\frac{x_i-y_i}{|x-y|^{d+1}}
\int_0^\infty \partial_j \omega \Big (x+ t \frac{x-y}{|x-y|}\Big)
( |x-y| +t)^{d}\, dt,
\endaligned
\end{equation}
\begin{equation}\label{A2-5}
\frac{\partial}{\partial y_j} K_i (x, y)=
-\frac{\partial}{\partial x_j}K_i (x, y)
+\frac{x_i -y_i}{|x-y|^d}
\int_0^\infty \partial_j \omega \Big( x+ t \frac{x-y}{|x-y|}\Big) 
\big( |x-y| + t)^{d-1}\, dt,
\end{equation}
and that $\text{\rm div} (u)= f$ in $\Omega$.
It follows by expanding the term $(|x-y| +t)^d$ in \eqref{A2-4}  that 
$$
\frac{\partial}{\partial x_j} K_i  (x, y) = \frac{k_{ij} (x, x-y)}{|x-y|^d}  + G_{ij} (x, y), 
$$
where $k_{ij}(x, z)$ and $G_{ij}(x, y)$ satisfy the conditions, $ |k_{ij} (x, z)|\le C$, $k_{ij}(x, tz)=k_{ij}(x, z)$ for $t>0$, 
$$
\int_{|z|=1} k_{ij} (x, z) dz =0 \quad \text{ and }  \quad
|G_{ij} (x, y) |\le C |x-y|^{1-d}.
$$
As a result, it  follows from the Calder\'on-Zygmund theory of singular integrals that 
$\| T_{ij} (f)  \|_{L^q(\Omega)} \le C \| f \|_{L^q(\Omega)}$. This gives  the estimate \eqref{A2-00}.

Finally, to prove \eqref{A2-0}, note that
 for $h=(h_1, h_2, \dots, h_d) \in C_0^\infty(\Omega; \R^d)$,
$$
\aligned
\Big|\int_\Omega u \cdot h \Big|
& = \Big| \int_\Omega f(y) \left\{ \int_\Omega K_i (x, y) h_i (x)\, dx \right\} dy \Big|\\
& \le \| f\|_{\hW^{-1, q}(\Omega)}
\| \nabla_y \int_\Omega K_i (x, y) h_i (x)\, dx \|_{L^{q^\prime} (\Omega)}.
\endaligned
$$
Moreover, 
$$
\frac{\partial }{\partial y_j}
\int_\Omega  K_i (x, y) h_i (x)\, dx
= h_i (y) a_{ij} (y) + \text{\rm p.v.}
\int_\Omega \frac{\partial}{\partial y_j} K_i (x, y) h_i (x)\, dx,
$$
where $|a_{ij} (y)|\le C $, and by \eqref{A2-5},
$$
\Big|\frac{\partial }{\partial y_j } K_i (x, y) + \frac{\partial }{\partial x_j} K_i (x, y) \Big|
\le C |x-y|^{1-d}.
$$
It follows that
$$
\aligned
\| \nabla_y \int_\Omega K_i (x, y) h_i (x)\, dx \|_{L^{q^\prime} (\Omega)}
 & \le C \| h\|_{L^{q^\prime} (\Omega)}
+ C  \sum_{i, j}  \| T_{ij}^* ( h_i) \|_{L^{q^\prime} (\Omega)}\\
& \le C \| h \|_{L^{q^\prime} (\Omega)}, 
\endaligned
$$
where $T_{ij}^*$ denotes the adjoint of $T_{ij}$. 
By duality we deduce that  $\| u \|_{L^q(\Omega)} \le C \|f \|_{\hW^{-1, q}(\Omega)}$.
\end{proof}

\begin{proof}[Proof of Theorem \ref{thm-A2}]

 For $r >0 $, define
$$
E_r  = \left\{ (x^\prime, x_d): \ |x^\prime|< r   \text{ and }  \psi (x^\prime) < x_d < C_0 r   \right\}.
$$
There exist $C_0\ge 1$ and $c_0\in (0, 1)$, depending only on $d$ and $M$, such that  $E_r$ is star-like with respect to 
every point in the ball centered at $(0, (1/2) C_0 r)$ with radius $c_0r$ \cite {JK-1982}.
By Lemma \ref{lemma-A2-1} there exists  a function $u_\ell \in W^{1, q}_0(E_\ell; \C^d )$ for $\ell \ge 1$  such that
$$
\text{\rm div}(w_\ell ) = g -\fint_{E_\ell} g   \quad \text{ in } E_\ell,
$$
\begin{equation}\label{T1-2}
\| \nabla w_\ell \|_{L^q(E_\ell)} \le C \| g \|_{L^q(E_\ell)}, 
\end{equation}
and 
\begin{equation}\label{T1-3}
\| w_\ell \|_{L^q(E_\ell)} \le C \| g \|_{\hW^{-1, q}(E_\ell)},
\end{equation}
where $C$ depends only on $d$, $q$ and $M$.

Next, we extend $w_\ell$ to $\H_\psi$ by zero.
Note  that 
$$
\|  w_\ell \|_{W_0^{1, q}(\H_\psi)} \le C  \left\{ \| g \|_{L^q(E_\ell)} + \| g \|_{\hW^{-1, q}(\H_\psi)} \right\}.
$$
It follows that there exists a subsequence, still denoted by $\{ w_\ell \}$ such that
$w_\ell \to w$ weakly in $W_0^{1, q}(\H_\psi; \C^d)$ for some $w\in W_0^{1, q}(\H_\psi; \C^d)$.
Moreover, the estimate \eqref{T1-0} holds. 
To show $\text{\rm div}(w) =g$ in $\H_\psi$, 
let $\varphi \in C_0^\infty(\H_\psi)$. Note that if $\ell$ is large,
$$
\int_{\H_\psi} w_\ell  \cdot \nabla \varphi = -\int_{\H_\psi} \Big( g-\fint_{E_\ell} g \Big) \varphi.
$$
 By letting $ \ell \to \infty$, we obtain 
 $$
 \int_{\H_\psi} w \cdot \nabla \varphi = -\int_{\H_\psi}  g  \varphi,
$$
where we have used the fact $\fint_{E_\ell} g \to 0$ for $g \in L^q(\H_\psi)$.
This completes the proof.
\end{proof}

%%%%%%%%%%%%%%%%%%%%%%

\noindent{\bf Conflict of interest.}  The authors declare that there is no conflict of interest.

\medskip

\noindent{\bf Data availability. }
 Data sharing not applicable to this article as no datasets were generated or analyzed during the current study.
 
\medskip

\noindent{\bf Acknowledgement.}
The authors thank the anonymous referees for helpful comments that improved the quality of the manuscript.

 \bibliographystyle{amsplain}
 
% \bibliography{Geng-Shen2024.bib}
 \bibliography{Geng-Shen2024b.bbl}

\smallskip

\begin{flushleft}

Jun Geng, School of Mathematics and Statistics, Lanzhou University, Lanzhou, P.R. China.\\
\emph{E-mail}:  \texttt{gengjun@lzu.edu.cn}

\medskip

Zhongwei Shen (corresponding author), Institute for Theoretical Sciences, Westlake University,
No. 600 Dunyu Road, Xihu District, Hangzhou, Zhejiang, 310030, P.R. China.\\
\emph{E-mail}: \texttt{shenzhongwei@westlake.edu.cn} \\

\end{flushleft}

\bigskip

\medskip

\end{document}